\documentclass{amsart}
\usepackage[utf8]{inputenc}
\usepackage{amsmath}
\usepackage{amsfonts}  
\usepackage{amssymb}
\usepackage{latexsym, amsmath, amssymb, amsfonts, amscd,bm}
\usepackage{amsthm}
\usepackage{t1enc}
\usepackage[mathscr]{eucal}
\usepackage{indentfirst}
\usepackage[]{graphicx}
\usepackage{pb-diagram}
\usepackage{fancyhdr}
\usepackage{fancybox}
\usepackage{enumerate}
\usepackage{color}
\usepackage{tikz}
\usepackage{tikz-cd}
\usepackage[all]{xy}
\usepackage{geometry}
\usepackage{setspace}
\usepackage[shortlabels]{enumitem}
\usepackage{mathtools}
\usepackage[style=alphabetic,isbn=false,maxcitenames=7, maxbibnames=5, giveninits=true]{biblatex}
\usepackage{comment}
\usepackage{enumitem}
\usepackage{hyperref}
\hypersetup{colorlinks=true,
linkcolor=black,          
    citecolor=black,        
    filecolor=magenta,      
    urlcolor=cyan,           
    }

\newtheorem{theorem}{Theorem}[section]

\newtheorem{prop}[theorem]{Proposition}
\newtheorem{corol}[theorem]{Corollary}

\newtheorem{innercustomthm}{Theorem}
\newenvironment{customthm}[1]
{\renewcommand\theinnercustomthm{#1}\innercustomthm}
{\endinnercustomthm}

\newtheorem{innercustomprop}{Proposition}
\newenvironment{customprop}[1]
{\renewcommand\theinnercustomprop{#1}\innercustomprop}
{\endinnercustomprop}

\newtheorem{innercustomcorol}{Corollary}
\newenvironment{customcorol}[1]
{\renewcommand\theinnercustomcorol{#1}\innercustomcorol}
{\endinnercustomcorol}

\theoremstyle{definition}
\newtheorem{defi}[theorem]{Definition}
\newtheorem{lemma}[theorem]{Lemma}

\theoremstyle{definition}
\newtheorem{example}[theorem]{Example}

\newtheorem{corollary}[theorem]{Corollary}
\theoremstyle{remark}
\newtheorem{remark}[theorem]{Remark}

\newcommand{\pairing}{\langle \: \cdot \: , \: \cdot \: \rangle}
\newcommand{\reduced}[1]{\mathfrak{a}_{\textup{ab}}(#1)}
\newcommand{\reduceds}[1]{\mathfrak{a}_{\textup{ab}}(#1)}
\newcommand{\bracksec}[1]{\mathfrak{a}(#1)}
\newcommand{\sections}[1]{\Gamma(\bigwedge^{\! \bullet} #1)}

\newcommand{\sectionsa}[1]{\mathfrak{a}(#1)}
\newcommand{\sectionsab}[1]{\mathfrak{a}_{\text{ab}}(#1)}

\newcommand{\cohomology}[2]{H^{#1}(\sectionsab{#2})}

\newcommand{\quo}[1]{`#1'}
\newcommand{\lieg}{\mathfrak{g}}

\newcommand{\mg}[1]{\mathfrak{g}}
\newcommand\restr[2]{{
  \left.\kern-\nulldelimiterspace 
  #1 
  \vphantom{\big|} 
  \right|_{#2} 
  }}

\calclayout

\makeatletter
\def\l@subsection{\@tocline{2}{0pt}{2.5pc}{5pc}{}}
\makeatother

\begin{document}

        \title[On the Dirac complement problem]{On the Dirac complement problem}
\author{Tom Ariel}
\email{tom.ariel@kuleuven.be}
\email{marco.zambon@kuleuven.be \vspace{-.2cm}}
\address{KU Leuven, Department of Mathematics, Celestijnenlaan 200B box 2400, 3001 Leuven, Belgium.}
\author{Roberto Rubio}
\email{roberto.rubio@uab.es\vspace{-.2cm}}
\address{\vspace{-.2cm}Universitat Aut\`onoma de Barcelona, 08193 Barcelona, Spain}
\author{Marco Zambon}

\begin{abstract}
The existence of a Dirac complement for a given Dirac structure is a central question in the structure theory of Courant algebroids and the deformation theory of Dirac structures. We study this problem in detail, proving the unobstructedness of lagrangian or local Dirac complements and providing  examples that show the complexity of this question. We introduce a cohomology class whose nonvanishing prevents the existence of a Dirac complement and apply it to several families of examples. On the other hand, by using Lie-theoretical techniques, we prove that, for a Lie algebra $\lieg$ endowed with a definite form, the diagonal $\Delta$ in $\lieg \oplus \bar{\lieg}$ does not admit a complement unless $\lieg$ is abelian. This includes real compact {semisimple} Lie algebras with their Killing form.
 
\end{abstract}

\maketitle

\vspace{-.3cm}

\tableofcontents
\section{Introduction}

Dirac structures arose in \cite{courant-weinstein,courant1990dirac} as a way to cast in a unified framework two kinds of geometric structures on a manifold: closed 2-forms (whose integrability condition is linear) and Poisson bivector fields  (whose integrability condition is quadratic). In general, a Dirac structure is a subbundle of a Courant algebroid that is lagrangian and involutive, and a Courant algebroid is a vector bundle over a manifold $X$ with additional structure, see Section 2.1 for the precise definitions. 

First, Dirac structures in the Courant algebroid $TX\oplus T^*X$ were considered. 
Another prominent family of Courant algebroids is that of quadratic Lie algebras, that is,\ Lie algebras endowed with an ad-invariant pairing. These
are the Courant algebroids for which the base manifold $X$ is a point; their Dirac structures are precisely the lagrangian subalgebras.

In this paper we address the following question: 
\smallskip
\begin{center}
\fbox{
\textit{ In a given Courant algebroid $E$, which Dirac structures $L$ admit Dirac complements?}
}
\end{center}
\smallskip
We point out that a complementary lagrangian subbundle to $L$ exists if and only if  the pairing on $E$ has split signature, so we will assume this in the following.

\subsection*{Motivation}
We present two motivations for this question.

\smallskip

1) If a Dirac structure $L$ in a Courant algebroid $E$ admits a complement $M$ that is itself a Dirac structure, then, by \cite[Thm. 2.5, 2.6]{liu-weinstein-xu:1997},
\begin{itemize}
    \item the pair $(L,M\cong L^*)$ forms a Lie bialgebroid,
    \item the Courant algebroid $E$ is obtained as the double of this Lie bialgebroid.
\end{itemize}
Moreover,  assuming that $L$ is integrable as a Lie algebroid, $L$ integrates to a Poisson groupoid.

When the base manifold $X$ is a point, this recovers the fact that the {double} of a Lie bialgebra $(\mathfrak{g}, \mathfrak{g}^*)$, which is $\mathfrak{g}\oplus  \mathfrak{g}^*$ endowed with the natural pairing and a Lie algebra structure different than the direct sum, is a quadratic Lie algebra. A Lie bialgebra endows a group $G$ integrating $\mathfrak{g}$ with the structure of a Poisson Lie group, as introduced by \cite{Drinfeld} and widely studied since.

In general, we see that the existence of a Dirac complement for a given Dirac structure is a sufficient condition, although not necessary as we show in Example \ref{ex:noDiraccomplement-still-a-double}, for the Courant algebroid to arise as the double of a Lie bialgebroid.

\smallskip

2) On the other hand, the deformation theory of a Dirac structure $L$ in a Courant algebroid $E$ is described by choosing an auxiliary complementary lagrangian subbundle $M$, so that deformations near $L$ (those staying complementary to $M$) are given by graphs of suitable maps $L\to M$.

If $M$ is Dirac, one can then construct a \emph{differential graded Lie algebra} (DGLA) whose  Maurer-Cartan  elements -- those satisfying the Maurer-Cartan equation -- parametrize the Dirac structures near $L$. The structure of the DGLA is quite simple; it consists of a differential - making it into a chain complex - and a compatible graded Lie bracket.

However, if $M$ is merely a lagrangian complement but not Dirac, then instead of a DGLA one obtains a  $L_3$ algebra, which is a more involved structure; in particular, there is one more structure map, and the binary bracket fails to satisfy the graded Jacobi identity. Thus, the existence of a Dirac complement ensures that the deformations of a Dirac structure is described in a simpler way.

\subsection*{Main results}

 We first show that the question above strictly deals with the integrability condition of Dirac structures, as lagrangian complements always exist, and, for Courant algebroids of the form $TX\oplus T^*X$, it is a global question, since it is always possible to find a local Dirac complement.

\begin{customprop}{\ref{prop:existencelagcomp}}
Any lagrangian subbundle of a Courant algebroid admits a lagrangian complement.
\end{customprop}

 \begin{customprop}{\ref{prop:existencelocalcomp}}
    In the exact Courant algebroid $(TX \oplus T^{*}X)_{H}$, for every Dirac structure $L$ and point $p \in X$, there exists an open neighborhood $U$ of $p$ and a Dirac structure $M$ in $(TU \oplus T^{*}U)_{i^{*}_{U} H}$ complementary to $L$, where $i_U : U \to X$ denotes the inclusion.
\end{customprop}

The best known example of Dirac structure that does not admit a Dirac complement (as we said, in split signature) is the cotangent bundle $T^*X$ in the Courant algebroid $TX\oplus T^*X$ twisted by a closed 3-form $H$ whose cohomology class is non-zero. Indeed, any lagrangian complement is necessarily the graph of a 2-form $\omega$, and the Dirac condition is equivalent to $d\omega=-H$; but since $H$ is not exact, there is no such $\omega$. This example inspires our first result, which we now describe.

Let $L$ be a Dirac structure in a Courant algebroid $E$ over a manifold $X$. 

Pick an auxiliary lagrangian complement $M$; the question is whether it is possible to deform $M$ to a \emph{Dirac} complement of $L$. We use the fact that the deformations of $M$ to a Dirac structure are parametrized by the Maurer-Cartan elements of a curved DGLA, that is, by solutions of the Maurer-Cartan equation: the inhomogeneous equation 
\begin{equation}\label{eq:introMC}
 N_{M} + d_{M} \omega + \frac{1}{2}[\omega,\omega]=0.
 \end{equation}
Here $N_M\in \Gamma(\bigwedge^{3} L)$ measures the failure of $M$ to be a Dirac structure; $d_{M}$ is a differential on $\Gamma(\bigwedge^{\bullet} L)$, induced by the  almost Lie algebroid $M$;  
$[\;,\;]$ is induced by the Lie algebroid bracket of $L$;   and 
the unknown $\omega$ lies in $\Gamma(\bigwedge^{2} L)$.

In Section \ref{sec:acohomologyfordirac}, we take the quotient $\reduced{L}$ of $\Gamma(\bigwedge^{\bullet} L)$ by its commutator subalgebra, in order to `remove' the quadratic term in \eqref{eq:introMC}.
We define the   
obstruction class $N^{L}$ as the class induced by $N_M$
in the cohomology of $\reduced{L}$ and show it is independent of the choice of lagrangian complement $M$. We then prove:

\begin{customthm}{\ref{theorem:obstrzero}}
    If $L$ has a Dirac complement, then $N^{L}=0$.
\end{customthm}
Therefore, if  $N^{L}\neq 0$, then there exists no  Dirac complement of $L$.
We provide various examples  in Section \ref{sec:obstructionclassapplications}, enlarging significantly the examples of Dirac structures that admit no complement. For instance, in Proposition \ref{prop:overapointcohomologycondition2} we provide sufficient conditions for a Dirac structure given by $L=I \oplus I^0 \subseteq (\lieg \oplus \lieg^*)_H$ for $I$ a Lie ideal not to admit a Dirac complement, and in Proposition \ref{poissonstrvanish} we provide sufficient conditions for the graph  of a twisted Poisson structure not to admit a Dirac complement. 
In  Proposition \ref{productnocomp}
we show that the converse of Theorem \ref{theorem:obstrzero}  does not hold. In Appendix \ref{sec:linfinityalgebras} we extend $N^L$ to  any curved DGLA or, more generally, any curved $L_{\infty}$ algebra.

 Finally, in Section \ref{sec:diracstructuresonliegroups}, we consider a Lie group $G$ with Lie algebra $\lieg$.  First, we consider the correspondence between Dirac structures on $(\mathfrak{g} \oplus \mathfrak{g}^{*})_{H}$ and left-invariant   Dirac structures on  $(TG \oplus  T^{*}G)_{{H}}$, for a Chevalley-Eilenberg closed $3$-form $H$ on $\lieg$ and its left-invariant extension. We give a Dirac structure in  $(\lieg \oplus \lieg^*)_{H}$ which does not admit a Dirac complement, yet the corresponding Dirac structure in $(TG \oplus  T^{*}G)_{{H}}$   does (Example \ref{ex:no-complement}).

More  interestingly, we assume the Lie algebra $\lieg$ is quadratic and we  consider the Courant algebroid $\lieg \oplus \bar{\lieg}$  and the Dirac structure given by the diagonal $\Delta$ (here   $\bar{\lieg}$ denotes the Lie algebra endowed with minus the pairing). 
The latter gives rise to the well-known Cartan-Dirac structure in $(TG\oplus T^*G)_{{\Omega}}$, for $\Omega$ the Cartan-Dirac $3$-form, by applying a certain isomorphism from
\cite{alekseev2007pure}.
When $\lieg$ is a complex semisimple Lie algebra equipped with the real part of its Killing form, a  Dirac complement of the Cartan-Dirac structure is given by proving that $\Delta$ admits such a complement \cite{vsevera2001poisson, alekseev2007pure}.

In contrast to this result, if $\lieg$ is equipped with a definite pairing,  we show in Theorem \ref{theorem:diagonalcompiffabelian} that $\Delta$ admits a Dirac complement if and only if $\lieg$ is abelian. As a corollary, we have the following result, which stands in contrast to the complex semisimple case (where the real part of the Killing form has split signature):

\begin{customcorol}{\ref{corol:cpctsemisimplenogo}}
    Assume $\lieg$ is a real, compact, semisimple Lie algebra, equipped with its Killing form $B$. Then the diagonal $\Delta$ in $\lieg \oplus \bar{\lieg}$ does not admit a Dirac complement. 
\end{customcorol}

\smallskip
\noindent\textbf{Notation:} 
Throughout this article we make use of the following notation:
\begin{itemize}
    \item $X$ will denote a manifold.
    \item $E$ will be a Courant algebroid, and $L,M$ will be lagrangian subbundles.
    \item $\mathfrak{g}$ will be a Lie algebra, $I$ will be an ideal.
    \item $A_i$ will denote sections of $TX$ or elements of $\mathfrak{g}.$
    \item $\alpha,\beta,\gamma...$ will denote sections of $T^{*}X$ or elements of $\mathfrak{g}^*$.
    \item $\mathcal{L}$ denotes the Lie derivative. 
    \item For a graded vector space $V$, by $V[k]$ we mean the graded vector space with grading given by $(V[k])_n=V_{n+k}$.
     For example, for the natural grading on $\Gamma(\bigwedge^{\bullet} T^{*}X)$, the degree 1 elements in $\Gamma(\bigwedge^{\bullet}T^{*} X) [1]$ are the  2-forms.

\end{itemize}
\smallskip

\smallskip
\noindent\textbf{Acknowledgements:} This project has been supported by MICIU/AEI/10.13039/501100011033 and the EU FEDER under the grants PID2022-137667NA-I00 (GENTLE) and CNS2024-154695 (DÉCOLLAGE). T.A. acknowledges support by FWO {fellowship} 11PSS24N. R.R. acknowledges support  from the MICIU/AEI and the EU FSE under the Ramón y Cajal fellowship RYC2020-030114-I and from the AGAUR under the grant 2021-SGR-01015. M.Z. acknowledges partial support by EOS project G0I2222N, FWO projects G0B3523N and G014726N  (Belgium), and Methusalem grant METH/21/03 - long term structural funding of the Flemish Government. The first and second authors would like to thank the Weizmann Institute of Science for its hospitality.

Finally, the second author thanks Henrique Bursztyn for first sharing with him the Example in Lemma \ref{lemma:tdualnocomp}, and the first author thanks Jonas Schnitzer for the useful discussion about Appendix~\ref{sec:linfinityalgebras}.

\section{Unobstructedness of lagrangian and local Dirac complements}\label{sec:def} 
We first recall basic definitions in the theory of Courant algebroids and Dirac structures, present the complement problem and prove two preliminary results. 

\subsection{Lagrangian and Dirac complements}
\label{subsec:compl}

    A \textbf{Courant algebroid} over a manifold $X$ consists of a vector bundle $E$ together with:
    \begin{itemize}
        \item a nondegenerate symmetric bilinear pairing $\langle \; \cdot \; , \; \cdot \; \rangle $ on $E$,
        \item a vector bundle map $\rho: E \to TX$, called the \textbf{anchor},
        \item a bilinear bracket  $[\; \cdot \; , \; \cdot \; ] : \Gamma(E)\otimes \Gamma(E) \to \Gamma(E)$ on the sections of $E$, 
    \end{itemize}
    satisfying the following compatibility conditions for all $e_{1},e_{2},e_{3} \in \Gamma(E)$:

    \begin{enumerate}[({A}1)]
    \item $[e_{1}, [e_{2},e_{3}]] = [[e_{1},e_{2}],e_{3}]+[e_{2},[e_{1},e_{3}]]$, \label{jacobiaxiom}
     \item $\rho(e_{1})\langle e_{2},e_{3}\rangle = \langle [e_{1},e_{2}],e_{3}\rangle + \langle e_{2},[e_{1},e_{3}]\rangle$,  \label{anchoraxiom}
        \item $[e_{1},e_{2}]+[e_{2},e_{1}]=\rho^{*}d\langle e_{1},e_{2}\rangle$.  \label{antisymmetryaxiom} 
    \end{enumerate} 
As a consequence of axioms \ref{jacobiaxiom}-\ref{antisymmetryaxiom}, we also have, for $f\in C^{\infty}(X)$:
    \begin{enumerate}[({A}1)]
    \setcounter{enumi}{3}
 \item $[e_{1},fe_{2}]=f[e_{1},e_{2}]+(\mathcal{L}_{\rho(e_{1})}f)e_{2}$\label{leibnizaxiom}.
     \end{enumerate}

    Over a point, $X=\{*\}$, a Courant algebroid is the same thing as a quadratic Lie algebra $(\mathfrak{g}, \langle \; \cdot \; , \; \cdot \;\rangle,[\; \cdot \;,\; \cdot \;], \rho) $, where the anchor $\rho$ is the only map $\mathfrak{g}\to \{*\}$. 

\smallskip
Over a manifold, probably the most relevant family of Courant algebroids is the following. 

\begin{example}\label{standardcourant}
    For a manifold $X$, the Courant algebroid $(TX\oplus T^{*} X)_{H}$ for $H\in\Omega^3_{cl}(X)$ consists of $E = TX \oplus T^{*}X$ together with: \begin{itemize}
        \item The pairing given for $x,y\in TX$ and  $\alpha,\beta\in T^{*}X$ by 
        $$ \langle x+\alpha,y+\beta\rangle :=\beta(x)+\alpha(y).$$ 
        \item The anchor given by the natural projection $\rho = p_{TX}: TX \oplus T^{*} X \to TX $. 
        
        \item The $H$-twisted Dorfman bracket $[\: \cdot \:, \: \cdot \: ]_H$: for $x,y \in \Gamma(TX)$, $\alpha,\beta \in \Gamma(T^{*} X)$,
        $$ [x+\alpha,y+\beta]_{H}= [x,y] + \mathcal{L}_{x} \beta - \iota_{y} d \alpha + \iota_{y} \iota_{x } H ,$$
        where $[x,y]$ is the usual Lie bracket. 
    \end{itemize} 
\end{example}

\begin{remark}
    Example \ref{standardcourant} exhausts all possible brackets giving the vector bundle $TX\oplus T^*X$, together with the natural pairing and projection, the structure of a Courant algebroid. These are, in fact, the models for the class of so-called \textbf{exact Courant algebroids} \cite[Let. 1]{vsevera2017letters}.
\end{remark} 

There is a class of Courant algebroids over a point $\{*\}$ that is analogous to Example \ref{standardcourant}.

  \begin{example}\label{gplusgdual}
 Let $\mathfrak{g}$ be a Lie algebra and $H$ a Chevalley-Eilenberg closed $3$-form on $\mathfrak{g}$. The Courant algebroid $(\mathfrak{g}\oplus \mathfrak{g}^{*})_{H}$ is given by the vector space $\mathfrak{g}\oplus \mathfrak{g}^{*}$ together with \begin{itemize}
        \item The natural pairing on $\mathfrak{g}\oplus \mathfrak{g}^{*}$, given for $A_1 + \alpha_1,A_2 + \alpha_2\in \mathfrak{g}\oplus \mathfrak{g}^{*} $ by: $$ \langle A_1 + \alpha_1, A_2 + \alpha_2 \rangle  = \alpha_1(   A_2) + \alpha_2(A_1).$$ 
        \item The anchor given by the zero map. 
        
        \item The bracket given, for $A_2,A_2\in \mathfrak{g}$ and $\alpha,\beta \in \mathfrak{g}^{*}$, by $$[A_1+\alpha, A_2+\beta]  = [A_1,A_2] + \mathcal{L}_{A_1} \beta - \mathcal{L}_{A_2} \alpha + \iota_{A_2} \iota_{A_1} H.$$ 
    \end{itemize} 
    Notice that in terms of the adjoint action, $\mathcal{L}_{A_1} \beta$ can be rewritten as $-(ad_{A_1})^*\beta$.
\end{example}

     In a Courant algebroid $E$, a subbundle $L \subseteq E$ is called \textbf{lagrangian} if $L=L^{\perp}$, where $\phantom{.}^\perp$ denotes the orthogonal with respect to the pairing. A \textbf{Dirac structure} is a lagrangian subbundle $L \subseteq E$ whose space of sections $\Gamma(L)$ is involutive under the bracket, that is, $[\Gamma(L),\Gamma(L)]\subseteq \Gamma(L)$. 

    The bracket and anchor restricted to a Dirac structure naturally give it the structure of a \textbf{Lie algebroid}, that is, a tuple $(L, [\; \cdot \; , \; \cdot \; ],\rho)$ consisting of a vector bundle $L$, a Lie bracket $[\; \cdot \; , \; \cdot \; ]$ on $\Gamma(L)$ and an anchor map $\rho:L\to TX$ satisfying  \ref{leibnizaxiom} for $e_i\in\Gamma(L)$.

        For a lagrangian subbundle $L\subseteq E$, involutivity under the bracket is equivalent to the vanishing of the \textbf{Nijenhuis tensor} given, for $l_{i} \in \Gamma(L)$, by
    \begin{equation} \label{eq:nijenhuis}
            N_{L}(l_{1},l_{2},l_{3}) = \langle [l_{1},l_{2} ] , l_{3} \rangle.
    \end{equation}

\begin{example}\label{graphof2form}
    For $\omega$ a $2$-form on $X$, let $L = graph(\omega) = \{  x+ \iota_{x} \omega  \; | \; x \in TX \}$. Then $L$ is a lagrangian subbundle of $(TX \oplus T^{*}X)_{H}$. 
    For $x,y\in \Gamma(TX)$, by developing the bracket, 
$$ [x+\iota_{x} \omega, y+\iota_{y} \omega ] =
         [x,y]+\iota_{[x,y]} \omega + (\iota_{y} \iota_{x} d_{dR}\omega + \iota_{y} \iota_{x} H).$$
    Therefore, for $x,y,z\in TX$, 
    $$N_{L} (x + \iota_{x} \omega, y+\iota_{y} \omega, z+\iota_{z} \omega) = (d_{dR}\omega + H)(x,y,z)$$ 
    and, from \eqref{eq:nijenhuis}, we see that $L$ is Dirac if and only if $d_{dR}\omega + H = 0$.
\end{example}

We introduce now our main objects of study.

\begin{defi}\label{def:complement}
    Let $E$ be a Courant algebroid. A \textbf{lagrangian complement} to a lagrangian subbundle $L\subseteq E$ is a lagrangian subbundle $L'\subseteq E$ such that $L \oplus L' = E$. The subbundle $L'$ is said to be a \textbf{Dirac complement} if $L'$ is Dirac. 
\end{defi}
By dimension considerations, two lagrangian subbundles are complements of each other if and only if $L \cap L'=\{0\}$ or, equivalently, if and only if $L+L'=E$. Therefore, for a pair of lagrangian subbundles, we will use the words \quo{complementary} and \quo{transverse} interchangeably. It should also be recalled that the pairing on $E$ gives rise to an isomorphism $L \cong (L')^*$, given by $\ell \mapsto \langle \ell,\: \cdot \: \rangle$. 
Our interest lies in the existence of Dirac or lagrangian complements, which is possible only under the following assumption.
\smallskip
\begin{center}
\fbox{
 \textit{Assumption: from now on, the pairing $\langle \; \cdot \;,\; \cdot \; \rangle$ of the Courant algebroid {is of signature $(n,n)$.}}
}
\end{center}
\smallskip

\subsection{The Dirac complement problem} The above discussion motivates the following natural question:
\smallskip
\begin{center}
\fbox{
\textit{ In a given Courant algebroid, when does a given Dirac structure admit Dirac complements?}
}
\end{center}
\smallskip
It is not always the case that a Dirac complement exists. To the best of our knowledge, the following provides the first, and only well-known,  instance of a Dirac structure that does not admit a Dirac complement.

\begin{lemma}
[{The case $L=T^*X$}]
\label{lemma:tdualnocomp}
    In the exact Courant algebroid $(TX\oplus T^{*}X)_{H}$ from Example \ref{standardcourant},  the Dirac structure $T^{*}X$ admits a Dirac complement if and only if $[H]_{dR} = 0$.     
\end{lemma} 
\begin{proof}
    Any lagrangian complement to $T^{*}X$ must be of the form $graph(\omega)$ where $\omega \in \Gamma(\bigwedge^{2}T^{*}X)$. As we showed in Example \ref{graphof2form}, this lagrangian complement is Dirac if and only if $d_{dR}\omega +H=0$, which is possible if and only if $[H]=0$.
\end{proof}
Over a point we get an analogous example:  
\begin{lemma}\label{gdualnocomp}

    Let {$(\mathfrak{g}\oplus \mathfrak{g}^{*})_{H}$} be the Courant algebroid from Example \ref{gplusgdual}. The Dirac structure $\{0\} \oplus \mathfrak{g}^{*}$ admits a Dirac complement if and only if $[H]_{CE} = 0$. 
\end{lemma}
\begin{proof}
    Just as in Lemma \ref{lemma:tdualnocomp}, any lagrangian complement of $\{0\} \oplus \mathfrak{g}^{*}$ is of the form $graph(\omega)$ for $\omega \in \bigwedge^{2} \mathfrak{g}^{*}$, and a similar calculation to that in Example \ref{graphof2form} shows that it is involutive if and only if $d_{CE}\omega + H =0$, where $d_{CE}$ is the Chevalley-Eilenberg differential. 
\end{proof}

We remark that while a \textit{given} Dirac structure in a Courant algebroid might not admit a Dirac complement, the Courant algebroid could still admit a pair of transverse Dirac structures.
\begin{example}\label{ex:noDiraccomplement-still-a-double}
    Consider $(TX\oplus T^{*}X)_{H}$ for $X = \mathbb{T}^{3}$ the three-torus with angular coordinates $\theta_1,\theta_2,\theta_3$ and $H$ any left-invariant 
    nonzero $3-$form,  i.e.\ $H=c\,d\theta_{1}\wedge d\theta_{2}\wedge d\theta_{3}$ for some non-zero $c\in \mathbb{R}$.
    Then $T^{*}X$ admits no Dirac complement, yet 
    \begin{align*}
    L &= \langle \partial_{\theta_{1}}, \partial_{\theta_{2}}, d\theta_{3}\rangle,& M&= \langle \partial_{\theta_{3}}, d\theta_{1}, d\theta_{2}\rangle     
    \end{align*}
     are a pair of transverse Dirac structures. 
\end{example}

In some Courant algebroids, there are no Dirac structures admitting a Dirac complement.
{In other words, not all Courant algebroids arise as the double of a Lie bialgebroid.}

\begin{example}\label{ex:so3nodiraccomps}
    Let $E = (\mathfrak{so}(3)\oplus\mathfrak{so}(3)^*)_H$, following the notation of Example \ref{gplusgdual}, where $H$ is any nonzero 3-form. A direct calculation shows that the only Dirac structures in $E$ are $\{0\} \oplus \mathfrak{so}(3)^*$ and $\langle A\rangle \oplus \langle A\rangle^{\circ}$ for any $A \in \mathfrak{so}(3)$, and it is easy to see that no two of these Dirac structures are complementary. Therefore, no Dirac structure in $E$ admits a Dirac complement. 
\end{example}

Finally, we present a Dirac structure not admitting a Dirac complement in an \textit{untwisted} exact Courant algebroid. We first recall several basic facts about $V \oplus V^*$ equipped with the usual pairing.
\begin{prop}[{\cite[Prop. 1.1.5]{courant1990dirac}}]\label{prop:representationoflagrangian}
    Let $L\subseteq V \oplus V^*$ be a lagrangian subspace, then $L$ admits a unique representation, for $F:= p_{V}(L)$ and $\varepsilon \in \bigwedge^2 F^*$, given by $$L= L(F,\varepsilon) := \{ X + \xi \; | \; X \in F, \; \iota_X \varepsilon = \restr{\xi}{F} \}.$$
\end{prop}

\begin{remark}\label{rem:integrabilityinLeepsilonrep}
    The representation $L(F,\varepsilon)$ of a lagrangian subspace can be applied for lagrangian subbundles of $TX \oplus T^*X$ as well, so a lagrangian subbundle defines a singular distribution $F$ and a $2$-form $\varepsilon$ on $F$ (due to the requirement for the lagrangian subbundle to be smooth, the converse is more delicate, that is, not any singular distribution $F$ and 2-form $\varepsilon$ will define a smooth lagrangian subbundle $L(F,\varepsilon)$). In this representation, integrability of $L(F,\varepsilon)$ gives that $F$ is an involutive distribution and $d_{F} \varepsilon + i_{F}^{*} H = 0$ (where $d_F$ is the leafwise de Rham differential) \cite{gualtieri-annals:2011}.
\end{remark}

This representation allows for a known characterization of transverse lagrangians.
\begin{lemma}\label{lemma:transversediracpoisson}
    Two lagrangian subspaces $L=L(F,\varepsilon), L'=L(F',\varepsilon')$ of $V \oplus V^*$ are complementary  if and only if $F + F'=V$ and $\restr{\varepsilon}{F\cap F'}-\restr{\varepsilon'}{F\cap F'}$ is nondegenerate. 
\end{lemma}
\begin{proof}
Note that
\begin{align*}
    L\cap L' = \{ X + \xi \; | \; X \in F \cap F', \; \restr{\iota_X\varepsilon}{F} = \restr{\xi}{F} \textrm{ and }  \restr{\iota_X \varepsilon'}{ F'} = \restr{\xi}{F'} \}.
\end{align*}

Hence, for $X\in V$, we have that $\iota_X(\restr{\varepsilon}{F\cap F'}-\restr{\varepsilon'}{F\cap F'})=0$ is equivalent to the existence of $\xi\in V^*$ such that $X+\xi\in L\cap L'$, which is moreover unique if $F+F'=V$. 

If $L$ and $L'$ are complementary, we have that $$V = p_V (V \oplus V^*) = p_V (L \oplus L') = p_V(L) +  p_V(L') = F + F',$$ and, by the equivalence above, $\restr{\varepsilon}{F\cap F'}-\restr{\varepsilon'}{F\cap F'}$ is non-degenerate. The converse also follows from the equivalence as, for $X+\xi\in L\cap L'$, we must have $X=0$ and, by the uniqueness, $\xi=0$. 

\end{proof}

As a corollary, we get the following \quo{parity formula} for complementary lagrangians:

\begin{prop}\label{prop:parityformulacomplements}

For $L=L(F,\varepsilon), L'=L(F',\varepsilon')$ complementary lagrangians in $V \oplus V^*$, we have $$\dim(F) + \dim(F') \equiv \dim(V)\pmod{2}.$$
\end{prop}
\begin{proof}
By Lemma \ref{lemma:transversediracpoisson}, $F\cap F'$ is even dimensional, as it admits a nondegenerate $2$-form, and 
\[ 
\dim(V ) = \dim (F+F')= \dim(F) + \dim(F') - \dim(F\cap F')\equiv \dim(F) + \dim(F') \pmod{2}.
\]
 
\end{proof}

Note that Lemma \ref{lemma:transversediracpoisson} and Proposition \ref{prop:parityformulacomplements} also hold, pointwise, for lagrangian subbundles of $TX \oplus T^*X$. We use them in our next example.

\begin{example}\label{ex:hopf}
    Let $p: S^3 \to S^2$ be the Hopf fibration, and consider $L = F \oplus F^{\circ} \subseteq TS^3 \oplus T^* S^3$, where $F = \ker (p_*) \subseteq TS^3$ is the rank-one vertical subbundle. Since $F$ is involutive, $L$ is a Dirac structure. Assume $L$ admits a Dirac complement $M$ and consider the distribution $F'=pr_{TS^3}(M)$. By Lemma \ref{lemma:transversediracpoisson}, $F+F' = TS^3$, so the rank of $F'$ at each point is either $2$ or $3$. By  Proposition \ref{prop:parityformulacomplements}, it must equal $2$ everywhere, so $F'$  is an involutive regular distribution complementary to $F$.
    Such an $F'$ would exactly be a flat \textit{Ehresmann} connection on the fiber bundle $p: S^3 \to S^2$, which is widely known to not exist (otherwise, since $S^2$ is simply connected, we would have $S^3 \cong S^1 \times S^2$), therefore, $L$ does not admit a Dirac complement. 
\end{example}

\begin{remark}
Given a Courant algebroid $E$, denote by $\mathcal{D}$ the space of Dirac structures, and by $\mathcal{D}_{C}$ the elements of $\mathcal{D}$ which admit a Dirac complement. When $E$ is over a point, the space of lagrangians in $E$ (and so the space of Dirac structures) carries a natural topology (see Remark \ref{rem:spaceoflag} below)  and it is clear that $\mathcal{D}_{C}$ is an open subset of $\mathcal{D}$. An interesting question is whether it is also a dense subset. In Example \ref{ex:so3nodiraccomps}, it is not, since $\mathcal{D}_{C}$ is empty. We do not know whether, whenever $\mathcal{D}_{C}$ is non-empty, it is  dense in  $\mathcal{D}$.
\end{remark}

\subsection{Non-obstructedness of the existence of a lagrangian complement}\label{nonobstructlag}

    As a first step in testing the existence of a Dirac complement, we show that finding a lagrangian complement is always possible. We first recall the notion of a generalized metric:

\begin{defi}\label{def:generalizedmetric}
    Let $E$ be a Courant algebroid. A \textbf{generalized metric} $(P,N)$ on $E$ is given by a pair of complementary subbundles $P, N \subseteq E$, such that the restriction of $\pairing$ to $P$ is positive definite, the restriction of $\pairing$ to $N$ is negative definite, and $P$ is orthogonal to $N$ with respect to $\pairing$. Under our conventions, it follows that the rank of both $P$ and $N$ is half the rank of $E$.

\end{defi}
\begin{remark}
    Any Courant algebroid admits a generalized metric  \cite[Sec. 1.6]{gualtieri-annals:2011}.
\end{remark}
Generalized metrics allow us to prove the existence of a lagrangian complement. We first need a known lemma for which we give  a coordinate-free proof that we will later use. 
\begin{lemma}\label{lemma:lagrangiansareorthogonalbundlemaps}(\cite[Sec. 1.2]{courant1990dirac})
    Let $E$ be a Courant algebroid and $(P,N)$ a generalized metric. Then we have a bijection between lagrangian subbundles $L \subseteq E$ and orthogonal bundle maps $\phi: P \to \bar{N}$ (where $\bar{N}$ is $N$ equipped with the opposite pairing), given by
    $$ 
    \phi \;\;\longleftrightarrow \;\; graph(\phi) := \{ p+ \phi (p) \in P \oplus N = E\; | \; p \in P\}.$$
\end{lemma}
\begin{proof}
    Let $\pi_{P} \colon E\to P$ and  $\pi_{N} \colon E\to N$ denote the  projections. Let $L \subseteq E$ be a lagrangian subbundle.  Note that $\ker(\restr{\pi_{P}}{L}) = L \cap N  = \{0\}$ since $L $ is lagrangian while $N $ is negative definite. 
    Similarly, $\ker(\restr{\pi_{N}}{L}) = \{0\}$. Since the ranks of $L$, $P$ and $N$ are equal,  $\restr{\pi_{P}}{L}$ and $\restr{\pi_{N}}{L}$ are invertible.
    
    Let $\phi: P \to N$ be the bundle map defined by $\phi = \pi_{N} \circ \restr{\pi_{P}}{L}\!\!\!^{-1}  $. We have that $$L = \{\restr{\pi_{P}}{L}\!\!\!^{-1}(p)\; | \; p \in P\}= graph(\phi) .$$ 
    Note that, since $L$ is lagrangian, for any $p \in P$, we have $0=\langle p+ \phi(p), p + \phi(p)\rangle = \langle p,p\rangle + \langle \phi(p),\phi(p)\rangle$, so $\phi$ defines a vector bundle isometry $$\phi: (P, \restr{\langle \; , \; \rangle}{P})   \to (N, -\restr{\langle \; , \; \rangle}{N}).$$
    Conversely, if $\phi$ is such an isometry, the above calculation shows that $L = graph(\phi)$ is lagrangian. 
\end{proof}

\begin{prop}\label{prop:existencelagcomp}

Any lagrangian subbundle of a Courant algebroid admits a lagrangian complement.

\end{prop}

\begin{proof} 
Let $L=graph(\phi)$ be a lagrangian subbundle of $E$. 
Consider $M := graph(-\phi)$. Note that
$$M \cap L = graph(\phi:P\to N) \cap graph(-\phi:P\to N) =  \{ p\in P \; | \; \phi(p) = -\phi(p)\} = \{0\},$$
and so $M$ is a lagrangian complement to $L$.
\end{proof}
The proof of Lemma \ref{lemma:lagrangiansareorthogonalbundlemaps} also shows the following:

\begin{corol}
Let $(P,N)$ be a generalized metric on $E$.    Any lagrangian subbundle is isomorphic, as a vector bundle, to $P$ (or $N$), and so, any two lagrangian subbundles of a Courant algebroid are indistinguishable as vector bundles.
\end{corol}
\begin{remark}\label{rem:spaceoflag}
   
    From the same proof we also obtain that the space of lagrangians in $E$, $Lag(E)$, is isomorphic, upon the choice of a generalized metric $(P,N)$, to the space of sections of pointwise-orthogonal operators $P \to N$. Over a point, this gives an identification of the space of lagrangians $Lag(E)$ with the manifold $O(n)$ (where $dim(E) = 2n$). One can verify that the induced topology on $Lag(E)$ is independent of the choice of generalized metric. 
\end{remark}

    Finally, the technique from the proof of Lemma \ref{lemma:lagrangiansareorthogonalbundlemaps} allows us to manufacture examples of Courant algebroids that do not admit lagrangian subbundles.
    \begin{example}\label{ex:no-lagrangian-subbundles}
 Consider any nontrivial vector bundle $V$ on a manifold $X$. Choose any metric $g_{1}$ on $V$. Take $E = V \oplus \mathbb{R}^{n}$ where $n=dim(V)$, with zero anchor and bracket, and pairing given by the metric $g_1$ on $V$ minus the Euclidean metric $g_{2}$ on the fibers of $X \times \mathbb{R}^{n}$. Following the proof of Lemma \ref{lemma:lagrangiansareorthogonalbundlemaps}, a lagrangian subbundle is now equivalent to a vector bundle isometry between $(V,g_{1})$ and $(\mathbb{R}^{n}, -g_{2})$. Since $V$ is nontrivial, no such isometry exists, so $E$ admits no lagrangian subbundles.
    \end{example}
    
    Therefore, in a general Courant algebroid, a lagrangian subbundle does not necessarily exist; however, if one exists, then there exists another one transverse to it, by Proposition \ref{prop:existencelagcomp}.

\subsection{Non-obstructedness of local existence of a Dirac complement} We show that in an exact Courant algebroid, there is no obstruction to \textbf{locally} finding a Dirac complement. 
Recall first the definition of a $B$-field transformation:

\begin{defi}\label{def:B-field}
    For a 2-form $B$, we define the bundle map $e^{B}: TX\oplus T^{*}X \to TX \oplus T^{*}  X$  by $$e^{B}(A + \alpha) = A + \iota_{A}B + \alpha$$
where $A \in TX$, $\alpha \in T^{*}X$. 
\end{defi}
 The bundle map $e^{B}$ has the following properties: 
    \begin{itemize}
        \item It is orthogonal with respect to the pairing, so it maps lagrangian subbundles to lagrangian subbundles.
        \item $e^{B}$ defines an isomorphism between the Courant algebroids $(TX \oplus T^{*}X)_{H}$ and $(TX \oplus T^{*}X)_{H+dB}$, that is, a vector bundle isomorphism which preserves the anchor, the pairing, and the bracket. If $B$ is closed, $e^B$ is an automorphism.
    \end{itemize}
\begin{prop}\label{prop:existencelocalcomp}
    In the exact Courant algebroid $(TX \oplus T^{*}X)_{H}$, for every Dirac structure $L$ and point $p \in X$, there exists an open neighborhood $U$ of $p$ and a Dirac structure $M$ in $(TU \oplus T^{*}U)_{i^{*}_{U} H}$ complementary to $L$, where $i_U : U \to X$ denotes the inclusion.
\end{prop}

\begin{proof}
    Take a contractible open neighborhood $U'$ of $p$, in which $H$ is exact. Choose a primitive $B \in \Omega^{2}(U')$, so that $H=dB$ in $U'$. By applying the Courant algebroid isomorphism $e^{-B}$, we can assume without loss of generality that $H=0$. Choose any lagrangian subspace $M$ in {$T_{p} U' \oplus T^{*}_{p} U'$} complementary to $\restr{L}{p}$ and, choosing coordinates in $U'$, extend it \textit{constantly} (in a translation invariant way) to all of $U'$. An easy argument shows that this gives a Dirac structure in \mbox{$(TU'\oplus  T^{*} U')_{0} $} (it is spanned by constant sections of $TU'\oplus  T^{*} U'$ relative to the coordinates chosen on $U$). 
    
    The Dirac structure $M$  is complementary to $L$ at $p$, so it is complementary to  $L$ on a possibly smaller neighbourhood $U\subseteq U'$ of $p$. 
\end{proof}

\section{Curved differential graded Lie algebras and deformation theory}\label{sec:cdgla}

In this section we recall basic definitions in the theory of curved differential graded Lie algebras and present their relation to the deformations of  lagrangian subbundles and Dirac structures. 

\subsection{Curved differential graded Lie algebras}\label{subsec:curved}
Curved differential graded Lie algebras, or curved DGLAs, are a generalization of standard (or \quo{flat}) DGLAs by the addition of a curvature term, which obstructs the differential from squaring to zero. 
\begin{defi}\label{def:curveddgla}
    A \textbf{curved differential graded Lie algebra} $(\mathfrak{a},R,\ell_1,\ell_2)$ is given by the following:
    \begin{itemize}
        \item A $\mathbb{Z}-$graded vector space $\mathfrak{a} = \oplus_{j\in \mathbb{Z}} \mathfrak{a}^{j}$.
           \item A degree $0$ map $\ell_2: \mathfrak{a}^{j}\otimes \mathfrak{a}^{k} \to \mathfrak{a}^{j+k}$, called the bracket. 
        \item A degree $+1$ map $\ell_1: \mathfrak{a}^{j}\to \mathfrak{a}^{j +1}$, called the differential.
     
        \item A degree $+2$ element $R \in \mathfrak{a}^{2}$, called the curvature.
    \end{itemize}
    They satisfy the following compatibility conditions, for any homogeneous $x,y,z \in \mathfrak{a}$ of degrees $|x |, |y|, |z|$:
    \begin{enumerate}
        \item $\ell_2$ is a graded Lie bracket, that is, the bracket is graded skew-symmetric:
        $$ \ell_2(x,y)  = -(-1)^{|x||y|}\ell_2(y,x),$$
        and satisfies the graded Jacobi identity:
        $$ \ell_2(x,\ell_2(y,z)) =\ell_2(\ell_2(x,y),z)+ (-1)^{|x||y|}\ell_2(y,\ell_2(x,z)) . $$
        \item The map $\ell_1$ is a graded derivation of the bracket: 
        $$ \ell_1 (\ell_2(x, y)) = \ell_2(\ell_1 x, y) + (-1)^{|x|} \ell_2(x,\ell_1 y).$$
        \item The curvature is closed under the differential, that is:
        $$ \ell_1(
        R)=0.$$
        \item The failure of $\ell_1$ to be a differential is controlled by the curvature: 
        $$ \ell_1^{2} = \ell_2(R, \; \cdot \;) .$$
         
    \end{enumerate}
\end{defi}

We introduce Maurer-Cartan elements, which will later be shown to correspond to deformations. 
\begin{defi}\label{defi:maurercartanelement}
    Let $(\mathfrak{a},R,\ell_1,\ell_2)$ be a curved DGLA. A \textbf{Maurer-Cartan element} is a degree $1$ element $\omega \in \mathfrak{a}^{1}$ satisfying the \textbf{curved Maurer-Cartan equation}:
    $$ 0 = R+ \ell_1 (\omega) + \frac{1}{2} \ell_2 (\omega,\omega). $$
\end{defi}


\subsection{The curved DGLA governing deformations of a lagrangian subbundle}\label{defthlinfty}
To discuss the Dirac complement problem, we recall the algebraic description of all possible lagrangian complements to a given lagrangian subbundle $L$. By choosing an arbitrary lagrangian complement $M$ (which always exists, by Proposition \ref{prop:existencelagcomp}), any other complementary {subbundle} to $L$ is given by the graph of some bundle map $\omega: M \to L$, that is, by the subbundle $\{ m + \omega(m) \; | \; m \in M \} \subseteq M\oplus L = E$. This subbundle is lagrangian when, through the identification $M \cong L^*$, we have $\omega \in \Gamma(\bigwedge^{2} L)$. This establishes the bijection
\[ \Gamma(\bigwedge\nolimits^{\!2} L) \leftrightarrow \{ \text{Lagrangian  complements   to  $L$}\}.\]

Our key idea along this section will be to choose, for a given Dirac structure $L$, an auxiliary lagrangian complement $M$ and attempt to deform it into a Dirac structure.

A pair $L$, $M$ of a Dirac structure $L$ and a lagrangian complement $M$ induces on the space $\Gamma(\bigwedge^{\bullet}L)[1]$ (recall the degree $1$ shift to make the bracket additive in degree -- so, for example, sections of $\bigwedge^{2}L$ are degree $1$ elements) the structure of a curved DGLA, which we apply below to the problem of deforming $M$.

\begin{prop}[{\cite[Cor. 3.7]{10.1093/imrn/rny134}}]
\label{prop:gms}

    Let $L$ be a Dirac structure and $M$ a lagrangian complement, so $L \cong M^{*}$ through the pairing. Then $\Gamma(\bigwedge^{\bullet}L)[1]$ carries the structure of a curved DGLA $(\Gamma(\bigwedge^{\bullet}L)[1], R, \ell_1, \ell_2)$, given by the following brackets:
    \begin{itemize}
        \item $\ell_{2}$ is given for $l_1, l_2 \in \Gamma(L)$ by the usual Courant bracket:
        $$\ell_2(l_1,l_2)=[l_{1},l_{2}] $$
       and extended to $\Gamma(\bigwedge^{\bullet}L)$ through the Schouten bracket formula: $$ [l_1 \wedge...\wedge l_p, l'_1\wedge...\wedge l'_q] = \sum_{1\leq i \leq p, 1\leq j \leq q} (-1)^{i+j} [l_i,l'_j]\wedge l_1\wedge...\wedge \hat{l}_i \wedge... \wedge l_p \wedge l'_1 \wedge ...\wedge \hat{l}'_j \wedge...\wedge l'_q.$$ 
       \item $\ell_{1} = d_{M}$, the quasi Lie algebroid differential of $M$ given, for $m_{i}\in\Gamma(M)$ and $ l \in \Gamma(L)$, by 
       \begin{equation}\label{eq:lagcompdm}
           (d_{M} l)(m_{1},m_{2})=\rho(m_{1})\langle l,m_{2}\rangle - \rho(m_{2})\langle l,m_{1}\rangle -\langle l, [m_{1},m_{2}]\rangle,
       \end{equation}  and extended as a graded derivation to $\sections{L}$. When $M$ is Dirac, this is the usual Lie algebroid exterior derivative on $\bigwedge^{\bullet}M^{*} \cong \bigwedge^{\bullet} L$. 
       \item $R=N_{M} \in \Gamma(\bigwedge^{3} M^{*}) \cong \Gamma(\bigwedge^{3} L)$ is the Nijenhuis tensor of $M$ as in Equation \eqref{eq:nijenhuis}.
    \end{itemize}
\end{prop}

\begin{remark}\label{rem:sign-change}
    Notice that our notion of  a curved DGLA in Definition \ref{def:curveddgla} differs from the one used in \cite{10.1093/imrn/rny134}. See Remark \ref{rem-gms18} for more details.
\end{remark}

\begin{remark}\label{rem:bracketonlagrangian}
    Let $p_M: E=L\oplus M \to M$ be the projection to $M$ along $L$. Then the lagrangian complement $L$ induces a quasi Lie algebroid bracket on sections of $M$, given for $m_1, m_2 \in \Gamma(M)$ by $$[m_1,m_2]_M = p_M([m_1,m_2]),$$ where $[m_1,m_2]$ is the bracket of $m_1,m_2$ as  sections of $E$.  The derivation $d_M$ can also be defined as the ``dual'' of this bracket, since $\langle l, [m_{1},m_{2}]\rangle=\langle l, p_M([m_{1},m_{2}])\rangle$. Of course, $[\; \cdot \;, \; \cdot \;]_M$ is generally not a Lie bracket, unless $M$ is Dirac, which would also make the projection redundant. 
\end{remark}

\begin{example}
[{The case $L=T^*X$}]
\label{ex:curveddglaforcotangent}
    For $TX\oplus T^*X$ in Example \ref{standardcourant}, consider the Dirac structure $L = T^{*}X$ and the lagrangian complement $M = TX$. The curved DGLA  $\Gamma(\bigwedge^{\bullet} T^{*} X)[1]$ is given by:
    \begin{itemize}
    \item $R=H$, since, for $x,y,z \in \Gamma(TX)$, 
    $$N_{TX}(x,y,z) = \langle [x,y]_{H},z\rangle = \langle [x,y] + \iota_{y} \iota_{x} H,z\rangle = \iota_{z} \iota_{y} \iota_{x} H = H(x,y,z).$$   
    
    \item $\ell_{1} = d_{dR}$,   since according to Formula \eqref{eq:lagcompdm}, $$p_{TX}([x,y]_{H}) = p_{TX} ([x,y] + \iota_{y}\iota_{x}H) = [x,y].$$  
    \item $\ell_{2} = 0$, since the bracket of two sections of $T^{*}X$ is zero.

\end{itemize}
\end{example}
As mentioned before, integrability of a lagrangian complement given by the graph of some $\omega \in \Gamma(\bigwedge^{2} L)$ is encoded by $\omega$ satisfying a certain equation in the curved DGLA $\Gamma(\bigwedge^{\bullet} L)[1]$.
\begin{prop}[\cite{keller2007formal,10.1093/imrn/rny134}]\label{prop:linfinitystr}
   { Let $L$ be a Dirac structure and $M$ a lagrangian complement.} The Dirac complements of $L$ are in bijection with Maurer-Cartan elements in the curved DGLA described in Proposition \ref{prop:gms}, that is, sections  $\omega \in \Gamma(\bigwedge^{2} L)$ satisfying the curved Maurer-Cartan equation:
$$ 0 = N_{M} + d_{M} \omega + \frac{1}{2}[\omega,\omega].$$

\end{prop}
For the case where $M$ is Dirac, so $N_M=0$, Proposition \ref{prop:linfinitystr} goes back to \cite{liu-weinstein-xu:1997}. 
\begin{example}
[{The case $L=T^*X$}]
\label{ex:maurercartanforcotangent}
    For the curved DGLA in Example \ref{ex:curveddglaforcotangent}, the curved Maurer-Cartan equation is
    \begin{equation}\label{eq:maurercartanforcotangent}
        0 = d_{dR} \omega + H.
    \end{equation}
    The existence of a solution to this equation depends on the cohomology class of $H$ and recovers the  fact that a Dirac complement to $T^{*}X$ exists if and only if $[H]=0$ (Lemma \ref{lemma:tdualnocomp}). 

\end{example}

\section{An obstruction class for Dirac structures}\label{sec:acohomologyfordirac}

In this section, for a Dirac structure, we define a canonical cohomology class and prove that it provides an obstruction to the existence of a Dirac complement.

Consider the curved DGLA from Example \ref{ex:curveddglaforcotangent}. 
When choosing a different lagrangian complement $M$ to $T^{*}X$, which must be of the form $M=graph(\omega_0)$ for some 2-form $\omega_0$, an easy calculation (following Proposition \ref{prop:gms}) shows that $d_M$ is still equal to the usual de Rham differential, while the curvature $H$ transforms to $H + d_{dR} \omega_0$. Notably, both the cohomology of $d_{M}$ and the class of the curvature $H$ are independent of $\omega_0$. 


In full generality, however, the Maurer-Cartan equation is given by the addition of a quadratic term, as in Proposition \ref{prop:linfinitystr}:
\begin{equation}\label{eq:maurercartanmoregeneral}
    0 = N_M+ d_M \omega + \frac{1}{2} [\omega,\omega].
\end{equation}

A necessary condition for the existence of a solution is not as obvious as in Equation \eqref{eq:maurercartanforcotangent}.
In this section, we can show that passing to a quotient of the curved DGLA - its \textbf{abelianization} - will allow us to project Equation \eqref{eq:maurercartanmoregeneral} to an equation of the \textit{form} of Equation \eqref{eq:maurercartanforcotangent}, where a natural cohomological obstruction for the existence of a solution is clear. 

\subsection{An abelianized cohomology for Dirac structures}

We begin with the following generalization of Proposition \ref{prop:linfinitystr}: 

\begin{prop}[\cite{keller2007formal}]\label{prop:hownijenhuischanges}
    Let $L$ be a Dirac structure, $M$ a lagrangian complement, and $M'$ a different complement, given by $M'=graph(\omega)$ {for some $\omega \in \Gamma(\bigwedge^{2} L)$}.
    The Nijenhuis tensors of $M$ and $M'$ are related via
    $$ N_{M'} = N_{M} + d_{M} \omega + \frac{1}{2}[\omega,\omega].$$
\end{prop}
Since the Nijenhuis tensor $N_{M}$ of any lagrangian complement only changes by terms of a given form, our key idea is to quotient out by such terms, which will give an invariant whose nonvanishing is an obstruction to the existence of a Dirac complement. To turn this observation into a construction, we first need two lemmas.

\begin{lemma}\label{lemma:howthedifferentialchanges}
        For $L, M$ and $M'$ as in Proposition \ref{prop:hownijenhuischanges}, we have that: $$d_{M'} = d_{M} + [\omega, \; \cdot \; ].$$

    \end{lemma}
    \begin{proof}
        We first prove the following formula:
        For  $l \in \Gamma(L)$,  $\omega \in \Gamma(\bigwedge^{2} L)$, $m \in \Gamma(M)$: 
        \begin{equation}\label{eq:inclusionprojection}
 \iota_{m} [\omega, l] = [ \iota_{m} \omega, l] - \iota_{p_{M}([l,m])} \omega            
        \end{equation}
            where $p_{M}: E \to M$ is the projection with respect to the splitting $E = L \oplus M$.
            By linearity, it is enough to prove this for $\omega = l_{1} \wedge l_{2}$ where $l_{1},l_{2} \in \Gamma(L)$.
        Note that 
        $$[\iota_{m}\omega, l ] = [\langle l_{1},m\rangle  l_{2} - \langle l_{2},m\rangle l_{1} , l ] = \langle l_{1},m\rangle[l_{2},l] - (\rho(l) \cdot \langle l_{1},m\rangle) l_{2} - \langle l_{2},m\rangle [l_{1},l] + (\rho(l) \cdot\langle l_{2},m\rangle)l_{1}.
        $$
        By Axiom \ref{anchoraxiom} in the definition of a Courant algebroid, this equals
\begin{multline*}
    \langle l_{1},m\rangle[l_{2},l]  - \langle l_{2},m\rangle [l_{1},l] + \langle[l,l_{2}],m\rangle l_{1} + \langle l_{2} ,[l,m]\rangle l_{1} - \langle[l,l_{1}],m \rangle l_{2} - \langle l_{1},[l,m]\rangle l_{2} \\
     = \iota_{m} [l,l_{1} \wedge l_{2}] - \iota_{p_{M}([l,m])} l_{1} \wedge l_{2}        
\end{multline*}  
        where we use the fact that $L$ is lagrangian, so for any $l \in \Gamma(L), e \in \Gamma(E)$, $\langle l,e\rangle = \langle l,p_{M} (e) \rangle$.
        
       Using this we will prove the proposition. Since $d_{M},d_{M'} $ are derivations of the wedge product, and again by linearity, it is enough to prove this for sections of $L$. Note that for $m_{1},m_{2} \in \Gamma(M)$:
        $$ d_{M} l (m_{1},m_{2}) = \rho(m_{1}) \langle l, m_{2}\rangle - \rho(m_{2}) \langle l, m_{1}\rangle - \langle l, [m_{1},m_{2}]\rangle,$$
        while
        \begin{align*}
            d_{M'} l (m_{1},m_{2}) &= \rho(m_{1}+\iota_{m_{1}}\omega) \langle l, m_{2} + \iota_{m_{2}}\omega \rangle - \rho(m_{2} + \iota_{m_{2}}\omega) \langle l, m_{1}+\iota_{m_{1}}\omega\rangle \\ &\phantom{= {}} - \langle l, [ m+\iota_{m_{1}}\omega,m_{2} + \iota_{m_{2}}\omega]\rangle \\
            &= \rho(m_{1}+\iota_{m_{1}}\omega) \langle l, m_{2} \rangle - \rho(m_{2} + \iota_{m_{2}}\omega) \langle l, m_{1}\rangle - \langle l, [ m_{1} + \iota_{m_{1}}\omega,m_{2} ] + [ m_{1},\iota_{m_{2}}\omega ]\rangle, 
        \end{align*}
        where in the last equality we use that $L $ is Dirac, so $[\iota_{m_{1}} \omega, \iota_{m_{2}}\omega] \in \Gamma(L)$ and so the pairings of $\iota_{m_{1}} \omega , \iota_{m_{2}}\omega, [\iota_{m_{1}} \omega, \iota_{m_{2}}\omega] $ with $l$ are all $0$. Therefore, we find
        \begin{align*} (d_{M'} l - d_{M} l)(m_{1},m_{2}) &= \rho(\iota_{m_{1}}\omega) \langle l, m_{2} \rangle - \rho(\iota_{m_{2}}\omega) \langle l, m_{1} \rangle - \langle l, [m_{1},\iota_{m_{2}}\omega] + [\iota_{m_{1}}\omega,m_{2}] \rangle \\
         &= \rho(\iota_{m_{1}}\omega) \langle l, m_{2} \rangle - \rho(\iota_{m_{2}}\omega) \langle l, m_{1} \rangle + \langle l, [\iota_{m_{2}}\omega, m_{1}] - [\iota_{m_{1}}\omega,m_{2}] \rangle \\ &\phantom{= {}}- \rho(l)\omega(m_{1},m_{2}) \\
       &= \langle [\iota_{m_{1} }\omega, l], m_{2} \rangle - \langle [\iota_{m_{2}}\omega, l],m_{1}\rangle - \rho(l)\omega(m_{1},m_{2}),
        \end{align*}
        where the second equality is due to Axiom \ref{antisymmetryaxiom} and the third equality is due to Axiom \ref{anchoraxiom}.
        Using Equation \eqref{eq:inclusionprojection} we see that the above equals
        \begin{align*}
            \langle \iota_{m_{1}}[\omega,l] + &\iota_{p_{M}[l,m_{1}]} \omega, m_{2}\rangle - \langle \iota_{m_{2}}[\omega, l] +\iota_{p_{M}[l,m_{2}]} \omega,m_{1}\rangle - \rho(l)\omega(m_{1},m_{2})\\
            & = 2\langle[\omega,l], m_{1} \wedge m_{2} \rangle + \langle \omega, [l, m_{1} \wedge m_{2}]\rangle -\rho(l)\langle \omega, m_{1} \wedge m_{2} \rangle = \langle [\omega,l ], m_{1} \wedge m_{2} \rangle,
        \end{align*}
    where the last equality is due to the extension of Axiom \ref{anchoraxiom} to $\Gamma(\bigwedge^{\bullet}E)$ with the pairing given by the extension of the pairing on $\Gamma(E)$ and the graded bracket given by the extension of the bracket on $\Gamma(E)$ via the usual Schouten formula. 
    \end{proof}

\begin{remark}\label{remark:differentcomptstwist}
    The similarity between Proposition \ref{prop:hownijenhuischanges} and Lemma \ref{lemma:howthedifferentialchanges} stems from the fact that, together, they show how the curved DGLA structure on $\Gamma(\bigwedge^{\bullet}L)$ transforms depending on choice of complement. For a lagrangian complement $M$ and $\omega \in \Gamma(\bigwedge^2 L)$, the curved DGLA structure induced by the complement $M' = graph(\omega)$ is the so-called \emph{twist} by $\omega$ (see Definition \ref{defi:twistofcurveddgla} in the appendix) of the curved DGLA structure induced by $M$, as observed in \cite{iglesias2012universal}.
    For a review of the more general context of twisting $L_{\infty}[1]$ algebras, see \cite{kraft2023introduction}.
\end{remark}

\begin{lemma}\label{lemma:onquotientthereisdifferential}
{Denote  $\sectionsa{L} :=\sections{L}[1]$.}
    The operator $d_{M}$ descends to a differential $\tilde{d}_{M}$ on the quotient  
    $$  \reduceds{L}:= \cfrac{\sectionsa{L}} {[\sectionsa{L},\sectionsa{L}]},
    $$
     where $[\sectionsa{L},\sectionsa{L}]$ is the space of finite sums of elements of {the form $[l_1,l_2]$ for $l_1,l_2\in  \sectionsa{L}$.}

\end{lemma}
\begin{proof}
    Denote the projection of an element $\alpha \in \sectionsa{L}$ to the quotient as $\alpha \mapsto \underline{\alpha}$. Since $d_{M}$ is a graded derivation of the bracket as mentioned in Proposition \ref{prop:linfinitystr}, it descends to a linear operator on the quotient. From Proposition \ref{prop:gms}, $d^{2}_{M} (\alpha) = [N_{M}, \alpha]$, and so we get $\tilde{d}^{2}_{M} \underline{\alpha} = 0$.
\end{proof}
\begin{remark}
    The quotient $\reduced{L}$ inherits a grading from $\sectionsa{L}$, such that $\tilde{d}_{M}$ is again a degree $+1$ differential. 
\end{remark}
From Lemmas \ref{lemma:howthedifferentialchanges} and \ref{lemma:onquotientthereisdifferential}, we immediately find:
\begin{prop}\label{prop:dM}
    The operator $d_{M}$ descends to a canonical differential $d$ on $\reduceds{L}$, independent of choice of lagrangian complement $M$. 
\end{prop}
\begin{proof}
    Since $(d_{M'} - d_{M})(\alpha) = [\omega,\alpha]$, we see $(\tilde{d}_{M'} - \tilde{d}_{M})(\alpha)=\underline{[\omega,\alpha]}=0$, and so $d = \tilde{d_M}$ is independent of $M$. 
\end{proof}
For any Dirac structure $L$, we have cohomology groups $\cohomology{i}{L}$, given by the cohomology of the cochain complex $(\reduced{L},d) $.

\subsection{The obstruction class \texorpdfstring{$N^{L}$}{NL}}

In this subsection we present a cohomology class, called the obstruction class, in the cohomology {of the cochain complex $(\reduced{L},d) $} defined in the previous section, and prove several of its properties. 
\begin{prop}\label{prop:independenceofobstructionclass}
    The Nijenhuis tensor of any lagrangian complement $M$ defines a cohomology class in $\cohomology{3}{L}$, and this class does not depend on the choice of complement.
\end{prop}

\begin{proof}
    From Proposition \ref{prop:gms} along with Definition \ref{def:curveddgla}, we find that $\tilde{d} \underline{N_{M}} = 0$ so indeed $\underline{N_{M}}$ defines a cohomology class. For any other complement given by $M' = graph(\omega:M\to L)$, 
    from Proposition \ref{prop:hownijenhuischanges} we find    
    $$
    \underline{N_{M'}} = \underline{N_{M}} + \underline{d_{M} \omega } = \underline{N_{M}} + d \underline{\omega},$$ and so the cohomology  classes 
     of $ \underline{N_{M}}$ and $ \underline{N_{M'}}$
    agree.
\end{proof}

\begin{defi}
    The \textbf{obstruction class} of a Dirac structure $L$ is: $$N^{L} = [\underline{N_{M}}] \in \cohomology{3}{L},$$
    where $M$ is any lagrangian subbundle complementary to $L$.
\end{defi}

\begin{theorem}\label{theorem:obstrzero}
    If $L$ has a Dirac complement, then $N^{L}=0$.
\end{theorem}
\begin{proof}
    The Nijenhuis tensor of a Dirac complement $M$ is $N_{M} =0$, so $N^{L} = [\underline{0}]=0.$ 
\end{proof}
Therefore, whenever the obstruction class does not vanish, a Dirac complement does not exist. 
\begin{remark}\label{curvedLtwist}
    While the class $N^{L}$ of a Dirac structure $L$ is canonical, the curved DGLA structure on $\sections{L}$ \textbf{is not}, that is, different choices of a complement do not yield  isomorphic curved DGLAs, not even in the broader $L_\infty $ sense. This is in contrast to the (flat) $L_{3}[1]$ algebra governing deformations of $L$ (defined on $\Gamma(\bigwedge^{\bullet} M)$ where $M$ is a lagrangian complement), which is indeed canonical as proved in \cite{10.1093/imrn/rny134}. As mentioned in Remark \ref{remark:differentcomptstwist}, these algebras are only twist-related to each other. 
    
\end{remark}

\begin{remark}
    We could ask how the Lie algebroid cohomology of $L$ compares to the cohomology $\cohomology{\bullet}{L}$. While the Lie algebroid cohomology of $L$ only `detects' the structure of $L$ as a Lie algebroid, the cohomology $\cohomology{\bullet}{L}$ can be used to detect the structure of $L$ as a \textit{Dirac structure inside of a Courant algebroid}. Indeed, as we saw in Theorem \ref{theorem:obstrzero}, if the obstruction class in $\cohomology{\bullet}{L}$ does not vanish, then $L$ does not admit a Dirac complement, but any Lie algebroid is a Dirac structure admitting a Dirac complement in \textit{some} Courant algebroid (for example, add to it its dual with zero bracket).
\end{remark}

\subsection{Properties of \texorpdfstring{$N^{L}$}{NL}}

The obstruction class is generally too coarse, in the sense that  it may vanish for Dirac structures that do not admit a complement (see, for instance, Example \ref{productnocomp} below). This is sometimes due to the fact that the space $\reduced{L}$ is very often $\{0\}$, especially for regular Dirac structures. However, we still find a family of examples for which this obstruction detects the nonexistence of a complement: abelian Dirac structures.

We say that a Dirac structure is abelian when the Courant bracket vanishes identically on $\Gamma(L)$. By Axiom \ref{leibnizaxiom}, this can only happen when the anchor is zero, that is, when $L$, as a Lie algebroid, is a bundle of Lie algebras. 
\begin{example}
    For the Courant algebroid over a point $(\mathfrak{g} \oplus \mathfrak{g}^{*})_{H}$ for a closed $3$-form $H$, if $\mathfrak{h} \subseteq \mathfrak{g}$ is an abelian ideal (for example, $Z(\lieg)$) such that $H \in \bigwedge^{3} \mathfrak{h}^{0}$, then $\mathfrak{h} + \mathfrak{h}^{0}$ is an abelian Dirac structure.  
\end{example}
\noindent For abelian Dirac structures, since $\bracksec{L}=0$, we have that $\reduced{L}=\sectionsa{L}$, and so we get a well-defined and canonical differential $d :=d_{M}$ on $\Gamma(\bigwedge^{\bullet} L)$. In such a case, we have:
\begin{prop}\label{abeliancompequiv}
    An abelian Dirac structure $L$ admits a Dirac complement if and only if $N^{L}=0$. 
\end{prop}
\begin{proof}
    For $L$ abelian, the Maurer-Cartan equation from Proposition \ref{prop:linfinitystr} is simply
    $$ 0 = N_{M} + d_{M}\omega,$$
    and the reduced cohomology is  the usual cohomology of $d_{M}$, which is indeed a differential as $d^{2}_M = [ N_{M}, \; \cdot \; ] = 0$. Thus, there exists a solution if and only if $N_{M}$ is exact, that is, $N^{L}=0$. 
\end{proof}

Finally, we see that for Courant algebroids over a point, both Dirac complements and obstruction classes are well behaved with respect to products.

\begin{prop}\label{prop:propertiesofsum}
    Let $E_{1},E_{2}$ be Courant algebroids over a point, $L_{1},L_{2}$ Dirac structures in them. We have that:
    \begin{enumerate}
        \item $L_{1}, L_{2}$ admit Dirac complements in $E_{1},E_{2}$ respectively if and only if $L_{1} \oplus L_{2}$ admits a Dirac complement in $E_{1} \oplus E_{2}$.
        \item There is a natural embedding $\mu: \cohomology{k}{L_{1}} \oplus \cohomology{k}{L_{2}} \hookrightarrow \cohomology{k}{L_{1} \oplus L_{2}}$.
        \item Under this embedding, $\mu(N^{L_{1}}\oplus N^{L_{2}})=N^{L_{1}\oplus L_{2}}$.
    \end{enumerate}
\end{prop}
\begin{proof}
    See Appendix \ref{appendixoverapoint}.
\end{proof}

\begin{remark}
   Note that Proposition \ref{prop:propertiesofsum} does not generally hold for Courant algebroids over a manifold. For instance, with the notation from Example \ref{ex:noDiraccomplement-still-a-double}, consider the Courant algebroids
    \begin{align*}
        E_1 &= (T\mathbb{T}^3 \oplus T^*\mathbb{T}^3)_{d\theta_1 \wedge d\theta_2 \wedge d\theta_3},& E_2& = T\mathbb{R}^3 \oplus T^* \mathbb{R}^3,
    \end{align*} and their Dirac structures $L_1 = T^* \mathbb{T}^3$, which does not admit a Dirac complement, and $L_2 = T\mathbb{R}^3$. 
On $X= \mathbb{T}^3 \times \mathbb{R}^3$, consider the Courant algebroid {$(TX\oplus T^*X)_H$}, where $H$ is the pullback of $d\theta_1 \wedge d\theta_2 \wedge d\theta_3$.
We show that its Dirac structure  $L_1\oplus L_2\subset (TX\oplus T^*X)_H$ admits a Dirac complement. Recalling the representation $L(F,\varepsilon)$ from Proposition \ref{prop:representationoflagrangian} and the $B$-fields from Definition \ref{def:B-field}, define
    $$ M := e^{-x_1 dx_2 \wedge dx_3}L( \langle \partial_{\theta_1} + \partial_{x_1} , \partial_{\theta_2} + \partial_{x_2} , \partial_{\theta_3} + \partial_{x_3} \rangle, 0) .$$
    Note that $F := \langle \partial_{\theta_1} + \partial_{x_1},  \partial_{\theta_2} + \partial_{x_2}, \partial_{\theta_3} + \partial_{x_3}\rangle$ is an involutive foliation complementary to the distribution {$\langle \partial_{x_1},\partial_{x_2}, \partial_{x_3}\rangle = \rho(L_1 \oplus L_2)$},  so the corresponding lagrangians are transverse by Lemma \ref{lemma:transversediracpoisson}. On the other hand, $B = -x_1 dx_2 \wedge dx_3$  satisfies $\restr{dB}{F} + \restr{H}{F} = 0$, so, by Remark \ref{rem:integrabilityinLeepsilonrep}, the subbundle $M$ is  a Dirac complement to $L_1\oplus L_2$. 

\end{remark}

\section{Applications of the obstruction class $N^{L}$}
\label{sec:obstructionclassapplications}
In this section we calculate the obstruction class for various examples of Dirac structures. We show that in certain cases it does not vanish, so there does not exist a Dirac complement as shown in Theorem \ref{theorem:obstrzero}, and in certain cases it vanishes although a Dirac complement does not exist.

\subsection{Examples of calculations of $N^{L}$ for Courant algebroids over a point}

We first display two examples of Dirac structures $L$ for which $N^{L} \neq 0 $, implying that  $L$ has no Dirac complement.

\begin{example}\label{angularmomentum}
    Let $\mathfrak{g} = \mathfrak{so}(3)$ and consider the Courant algebroid over a point $\mathfrak{g}\oplus \mathfrak{g}^{*}$
with the bracket described in Example \ref{gplusgdual} (with $H=0$). 
    For any non-zero $W \in \mathfrak{g}$, let $L$ be the Dirac structure $\langle W \rangle +\langle W \rangle^{0} $. Then $N^{L} \neq 0 $, and so $L$ has no Dirac complement.
    
    To see why, we identify $\mathfrak{g} = \langle e_1,e_2,e_3 \; | \;[e_i,e_j]=\varepsilon_{ijk} e_k \rangle$ and $\mathfrak{g}^{*} =  \langle \alpha_1,\alpha_2,\alpha_3\rangle$, where $\{\alpha_i\}$ is the basis dual to $\{e_i\}$, and assume without loss of generality $W = e_{1}$, since we can get this through a rotation (which is a Lie algebra isomorphism), and so $L= \langle e_{1},\alpha_{2},\alpha_{3}\rangle$. 

        Consider the lagrangian complement $M = \langle e_2, e_3, \alpha_1 \rangle$. The brackets on $M$ are given by:
  \begin{align*}
                     [e_2, e_3] &= e_1, & p_M([e_2,e_3])&=0,\\
        [e_2, \alpha_1] &=\iota_{e_2} \alpha_2\wedge\alpha_3=\alpha_3,&  p_M([e_2, \alpha_1])&=0,\\
        [e_3, \alpha_1] &= \iota_{e_3} \alpha_2 \wedge \alpha_3= -\alpha_2,& p_M([e_3, \alpha_1] )&=0.
        \end{align*}       
        Therefore, following Proposition \ref{prop:gms} and Remark \ref{rem:bracketonlagrangian}, we have that $N_{M} =  e_1 \wedge \alpha_2 \wedge \alpha_3$ while $d_M=0$, so the obstruction class is given simply by the class of $N_M$ in $\reduced{L}$. 
        
        Furthermore, we show that $[\bigwedge^{2} L, \bigwedge^{2} L]=[\bigwedge^{3} L, L]  = \{0\}$, and so $(\reduced{L})^{3}$ is simply $\bigwedge^{3} L$.
        The brackets on $L$ are given by 
        \begin{align*}
        [e_{1},\alpha_{2}] &=\iota_{e_{1}}d\alpha_{2}=-\iota_{e_{1}} \alpha_{1}\wedge\alpha_{3}=-\alpha_{3},\\
        [e_{1},\alpha_{3}]&=\iota_{e_{1}}d\alpha_{3}=\iota_{e_{1}} \alpha_{1}\wedge\alpha_{2}=\alpha_{2},    
        \end{align*}
        so calculating the Schouten brackets on $\bigwedge^{\bullet}L$, we see for example:
        \begin{align*}
        [e_{1}, e_{1} \wedge \alpha_{2} \wedge \alpha_{3} ] & = \alpha_{3}\wedge e_{1} \wedge \alpha_{3}  + \alpha_{2} \wedge e_{1} \wedge \alpha_{2} = 0,\\
        [e_{1} \wedge \alpha_{2}, e_{1} \wedge \alpha_{3} ] & = \alpha_{3} \wedge e_{1} \wedge \alpha_{3}  -\alpha_{2} \wedge e_{1} \wedge \alpha_{2} = 0.
        \end{align*}
        One can verify that any other bracket of two elements of $\bigwedge^{2} L$ or an element of $L$ and an element of $\bigwedge^{3} L$ is also zero.         Since $0\neq N_{M} = N^{L}$, we have that $L$ in Example \ref{angularmomentum} does not admit a Dirac complement. 
    
\end{example}

\begin{example}\label{ex:ABZ}
    Let $\mathfrak{h} = \langle x,y,z \; | \; [x,y]=z\rangle$ be the Heisenberg Lie algebra. Let $E =\mathfrak{h} \oplus \mathfrak{h^{*}}$ with the Dorfman bracket
   described in Lemma \ref{gdualnocomp}, for \textit{any} $H\in \bigwedge^3 \mathfrak{h^{*}}$. 
 Then  $L=\langle z \rangle + \langle z \rangle^{0}$ has no Dirac complement.
 
Indeed, since $\langle z \rangle $ is both one-dimensional and the center of $\mathfrak{h}$, it is easy to verify that $L$ is an abelian Dirac structure for any twist $H$. Let $\alpha, \beta, \zeta$ be the dual basis to $x,y,z$, and denote $H = \lambda \alpha \wedge \beta \wedge \zeta$. Consider the lagrangian complement $M = \langle x,y,\zeta \rangle $. Note that the brackets on $M$ are given by:
\begin{align*}
  [x,y]&=z+\lambda\zeta,& p_M ([x,y])&=\lambda \zeta,\\ 
[x,\zeta]&=\beta, & p_M([x,\zeta])&=0,\\
[y,\zeta] &=-\alpha,& p_M([y,\zeta])&=0.  
\end{align*}

Therefore, following Proposition \ref{prop:gms} and Remark \ref{rem:bracketonlagrangian}, we have that $N_M = \alpha \wedge \beta \wedge z$ while $d_M$ is given by:
\begin{align*}
d_M \alpha &= 0,& d_M \beta &= 0,& d_M z &= - \lambda \alpha\wedge \beta.    
\end{align*}
\noindent Immediately we see that $d_M$ restricted to $\bigwedge^{2} L$ is zero, and again we have $0\neq N_M = N^L$ and $L$ in Example \ref{ex:ABZ} does not admit a Dirac complement. 
\end{example}

Even if $N^{L}$ is zero, a Dirac complement may not exist. To give an example, we introduce the following Courant algebroid over a point, which will be our focus  here until the end of this subsection, and in Section \ref{sec:diracstructuresonliegroups}.

\begin{defi}\label{defi:gplusgbar} 
Let $\lieg$ be a quadratic Lie algebra with pairing $\pairing$. Denote by $\bar\lieg$ the same Lie algebra, with pairing given by $-\pairing$, which is again a quadratic Lie algebra. Therefore, the orthogonal direct sum $\lieg \oplus \bar\lieg$ is a quadratic Lie algebra with a split signature pairing, and so is a Courant algebroid over a point. 
\end{defi} 
 
   The diagonal $\Delta = \{ (x,x) \; | \; x \in \lieg \}$ is a Dirac structure in $\lieg \oplus \bar\lieg$ (which was considered, for example, {in \cite{evens2006variety} and \cite{iglesias2012universal}}).

\begin{prop}\label{prop:obstructionofdiagvanishes}
    For the Dirac structure $\Delta \subseteq \lieg \oplus \bar\lieg$, we have that $N^\Delta =0$. 
\end{prop}
\begin{proof}
    Let $\Delta_{-} = \{ (x,-x) \; | \; x \in \lieg \}$. Then $\Delta_{-}$ is a lagrangian complement to $\Delta$, which is not Dirac. Let $\{ A'_i \}$ be an orthonormal basis of $\lieg$, which yields bases $A_i = (A'_i,A_i')$ and $B_i = (A'_i, -A_i')$ of $\Delta$ and $\Delta_-$ respectively. Let $\Gamma_{ij}^k $ be the structure constants of $\lieg$ with respect to $\{A'_i\}$. Note that since
    $$\langle A_i,B_j\rangle = 2\delta_{ij},\quad \quad [B_i,B_j] = ([A_i',A'_j],[A'_i,A'_j]) = \sum_k \Gamma_{ij}^k A_k,$$ we have that
    $$ N_{\Delta_-} \propto \sum_{i,j,k} A_i \wedge A_j \wedge  \Gamma^{k}_{ij} A_k  =  \sum_{i,j} A_i \wedge A_j \wedge [A_i,A_j] =  \sum_{i,j} [A_i \wedge A_j, A_i \wedge A_j]   \in [\bigwedge\nolimits^{\!2} \Delta, \bigwedge\nolimits^{\!2} \Delta],$$
    and so $\underline{N_{\Delta_-}}=0$. 
\end{proof}
Note that $\Delta \cong \lieg$ as Lie algebras, therefore a Dirac complement to $\Delta$ induces a Lie bialgebra structure on $\lieg$, which in turn induces a Lie-Poisson structure on any Lie group $G$ integrating $\lieg$ (this widely known fact goes back to Drinfeld, see \cite{drinfel1990hamiltonian}). In the following example, we show that, despite always having $N^\Delta=0$, the diagonal $\Delta$ may not admit a Dirac complement.

\begin{example}\label{ex:so3diagonal}
    Let $\lieg = \mathfrak{so}(3)$ endowed with the Killing form, and consider $E = \lieg \oplus \bar{\lieg} $. As we have seen, the diagonal $\Delta$ is a Dirac structure, and the anti-diagonal $\Delta_{-}$ is a lagrangian complement.  Following the notation in Example \ref{angularmomentum} and taking the orthonormal basis $A_i = (e_i,e_i)$ for $\Delta$, we have that $N_{\Delta_{-}}  = 2A_{1} \wedge A_{2} \wedge A_{3}$. Note that $[\Delta_-, \Delta_-] \subseteq \Delta$, and so $d_{\Delta_-} = 0$.

    Observing the Maurer-Cartan equation $N_{M} + \frac{1}{2} [\omega,\omega]=0$, we have that for an arbitrary element $\omega = \lambda_{1} A_{1} \wedge A_{2} + \lambda_{2} A_{2} \wedge A_{3} + \lambda_{3} A_{3} \wedge A_{1} \in \bigwedge^2 \Delta$, 
    $$  N_{M} +\frac{1}{2} [\omega,\omega] = (2 + \lambda_{1}^{2} + \lambda_{2}^{2} +\lambda_{3}^{2} ) A_{1} \wedge A_{2} \wedge A_{3}.$$
Therefore the Maurer-Cartan equation has no solution, so $\Delta$ does not admit a Dirac complement. 
\end{example}

\color{black}

\subsection{A necessary condition for $N^{L}=0$ in Courant algebroids over a point }
\label{subsec:nec}
Consider now the case $(\mathfrak{g}\oplus  \mathfrak{g}^{*})_{H}$, with $\lieg$ a Lie algebra, $I\subseteq \mathfrak{g}$ a Lie ideal, and $H \in \bigwedge^{3} I^{0}$ closed. Since $H \in \bigwedge^{3} I^{0}$, it descends to a 3-form $\tilde{H}$ on the quotient $\mathfrak{g}/I$, which carries a Lie algebra structure. Since $H$ is closed, so is $\tilde{H}$, and so its cohomology class $[\tilde{H}]$ in the Chevalley-Eilenberg cohomology of $\mathfrak{g}/I$ is well defined. The following proposition shows that nonvanishing of this class obstructs the existence of a Dirac complement to the lagrangian subalgebra $L = I + I^{0}$. 

\begin{prop}\label{prop:overapointcohomologycondition2}
    Let $I\subseteq \mathfrak{g}$ be a Lie ideal, and let $H \in \bigwedge^{3} I^{0} $ be a Chevalley-Eilenberg closed form, so it descends to a closed $\tilde{H} \in \bigwedge^{3} (\mathfrak{g}/I)^{*}$.  
    Consider the Dirac structure $L = I + I^{0}$ in $(\mathfrak{g}\oplus  \mathfrak{g}^{*})_{H}$.
    If $[\tilde{H}]_{\mathfrak{g}/I} \neq 0$, then $N^{L}\neq0$ and so $L$ admits no Dirac complement. 
\end{prop}
\begin{proof}
    We will relate the class $[\tilde{H}]$ to a cohomology class in a certain subalgebra of a special lagrangian complement. Choose an arbitrary complementary vector space $F$ to $I$ and consider the lagrangian complement $M = F \oplus F^{0}$. The projected bracket of two elements $f_{1}, f_{2} \in F$ with respect to the splitting $L \oplus M$ is given by: $$ [f_{1},f_{2}]_{M} = p_{M} ([f_{1},f_{2}] + \iota_{f_{1}} \iota_{f_{2}} H ) = p_{F} ([f_{1},f_{2}]),$$
    where $H \in \bigwedge^{3} I^{0}$ so $\iota_{f_{1}} \iota_{f_{2}} H \in I^{0} \cong F^{*}$, and so it vanishes when projecting to $M$. Therefore, $F \oplus \{0\} $ is a closed under the bracket on $M$, and it is not hard to verify it is isomorphic to $\mathfrak{g}/I$ as a Lie algebra. 

    Consider now the bracket on $L = I +I^{0}$. Since $I$ is an ideal, $d(I^{0}) \subseteq \bigwedge^{2} I^{0}$, therefore, for $e_{1} + \alpha_{1}, e_{2} + \alpha_{2} \in I\oplus I^{0} $, we have: $$ [e_{1} + \alpha_{1}, e_{2} + \alpha_{2} ] = [e_{1},e_{2} ] .$$
    Therefore, by the formula for the Schouten bracket, $[\bigwedge^{\bullet} L, \bigwedge^{\bullet} L ] \subseteq I \cdot \bigwedge^{\bullet} L$.

    Assume $L$ admits a Dirac complement, so there is some $\omega \in \bigwedge^{2} L$ such that:
    $$ 0 = N_{M}+ d_{M} \omega + \frac{1}{2} [\omega,\omega] $$
    Restricting to $F \oplus \{0\} \subseteq M$, we have $i_{F}^{*} N_{M} = i_{F}^{*} H$, $i_{F}^{*} d_{M} \omega = d_{F} i_{F}^{*} \omega$ since $F$ is a subalgebra, and $i_{F}^{*} [\omega,\omega] = 0 $ by the calculation of the bracket on $L$.
    We therefore get:
    $$  0 = i_{F}^{*} H + d_{F} i_{F}^{*} \omega,$$
    which is a contradiction since $[\tilde{H}] \neq 0$ and $F\oplus \{0\} \cong \mathfrak{g}/I$. 
\end{proof}
As an example for Proposition \ref{prop:overapointcohomologycondition2}, we have the following:
\begin{example}
    Consider $I= [\mathfrak{g},\mathfrak{g}]$. Then any $H \in \bigwedge^{3} I^{0}$ is closed, and since $\mathfrak{g}/I$ is abelian, any nonzero $H \in \bigwedge^{3} I^{0}$ has $[\tilde{H}]_{g/I} \neq 0 $. Therefore, for any nonzero $H \in \bigwedge^{3} [\mathfrak{g},\mathfrak{g}]^{0}$, the Dirac structure $[\mathfrak{g},\mathfrak{g}] \oplus [\mathfrak{g},\mathfrak{g}]^{0}$ does not admit a Dirac complement in $(\mathfrak{g}\oplus\mathfrak{g}^{*})_{H}$. 
\end{example} 
\begin{remark}

Proposition \ref{prop:overapointcohomologycondition2} can be understood also as a simple instance  of Courant algebroid reduction (see e.g. \cite[Thm. 6.11 , Thm. 7.11]{BCMZ}). In the setting of that proposition, let $K:=I\oplus \{0\}$, so its orthogonal w.r.t. the pairing is $K^{\perp}=\mathfrak{g}\oplus I^{\circ}$, and $K^{\perp}/K=(\mathfrak{g}/I)\oplus (\mathfrak{g}/I)^*$ inherits the Courant algebroid structure (over a point) twisted by $\tilde{H}$. The reduced Dirac structure $L_q:=(K^{\perp}\cap L)/K$ is just 
$\{0\}\oplus (\mathfrak{g}/I)^*$. If $M$ is any Dirac structure in $(\mathfrak{g}\oplus  \mathfrak{g}^{*})_{H}$ which is complementary to $L$, then the reduced Dirac structure $M_q$ is a \emph{complement} to $L_q$ (this holds as a consequence of $K\subset L$), yielding a contradiction to Lemma \ref{gdualnocomp}.

\end{remark}
\begin{remark}\label{remarkonsplit}
    Under the conditions of Proposition \ref{prop:overapointcohomologycondition2}, if $I$ admits a complementary subalgebra, we have again that $L=I \oplus I^{0}$ admits a Dirac complement, provided $N^{L}=0$. This is however \textbf{not} a necessary condition, since it could be that $I$ does not admit a complementary subalgebra while $I\oplus I^{0}$ admits a Dirac complement. 
\end{remark}

\subsection{The obstruction class $N^{L}$ in exact Courant algebroids}

We first display examples of Dirac structures whose obstruction class does not vanish; when then know that they do not admit a Dirac complement.
We start revisiting the example of the Dirac structure $T^*X$, from the point of view of the obstruction class $N^L$.

\begin{example}[{The case $L=T^*X$}]
    In the twisted exact Courant algebroid $(TX \oplus T^{*} X)_{H}$ for $[H] \neq 0 $, let $L = T^{*}X$. Then $N^{L} \neq 0$, recovering the well-known fact that $T^{*}X$ does not have a Dirac complement in this Courant algebroid, as shown in Lemma \ref{lemma:tdualnocomp}.
    Indeed, 
    since $[\Gamma(T^{*}X), \Gamma(T^{*}X) ] = \{ 0 \}$, we have $\reduced{L} = \Gamma(\bigwedge^{\bullet} L)$. Choosing as lagrangian complement $M = TX$, it is easy to see that $d_{M} = d_{dR}$ (despite the twist $H$), and so the cohomology $\cohomology{\bullet}{L}$ recovers $H^{*}_{dR} (X)$.
    Since $N_{M} = H$,  we have $N^{L} = [H]_{dR} \neq 0 $. 
\end{example}

The following proposition extends the previous example, allowing $L$ to be a (twisted) Poisson structure and requiring the existence of a submanifold (not necessarily the whole of $X$) satisfying a cohomological condition. {Recall that given a closed 3-form $H$ on $X$, a twisted Poisson structure  is a bivector field whose graph is a Dirac structure in $(TX \oplus T^{*} X)_{H}$ \cite{vsevera2001poisson}.}

\begin{prop}\label{poissonstrvanish}
    In the twisted exact Courant algebroid $(TX \oplus T^{*} X)_{H}$, let $L$ be the graph of a twisted Poisson structure $\pi$. Assume $\pi$ vanishes on a submanifold $Y \subseteq X$ such that $[i^{*}_{Y} H] \neq 0 $ in $H^{3}_{dR}(Y)$. Then $N^{L} \neq 0 $, and so $L$ has no Dirac complement. 
\end{prop}
\begin{proof}
    For $\alpha ,\beta \in \Gamma(T^{*} X)$: $$ [\alpha + \pi^{\#} \alpha , \beta +\pi^{\#} \beta   ] = \mathcal{L}_{\pi^{\#} \alpha} \beta - \mathcal{L}_{\pi^{\#} \beta} \alpha + d_{dR}(\pi(\alpha,\beta)) + \pi^{\#}(\mathcal{L}_{\pi^{\#} \alpha} \beta - \mathcal{L}_{\pi^{\#} \beta} \alpha + d_{dR}(\pi(\alpha,\beta))) .$$ Note that: \begin{itemize}
    \item The expression $\pi(\alpha,\beta)$ vanishes along $Y $, 
    hence $d_{dR}(\pi(\alpha,\beta)) \in \Gamma(TY^{0})$.
    \item By Cartan's formula, $\mathcal{L}_{\pi^{\#} \alpha} \beta = d_{dR}(\pi(\alpha,\beta)) + \iota_{\pi^{\#} \alpha} d\beta$. The first term is again lies in $\Gamma(TY^{0})$, and the second vanishes along $Y.$
\end{itemize}
Overall, identifying $T^{*}X$ with $L$ using the pairing, we find that $[\Gamma(L),\Gamma(L)] \subseteq \Gamma(TY^{0})$. By the Leibniz rule for the Schouten bracket, we have $$[\Gamma(\bigwedge\nolimits^{\!\bullet} L), \Gamma(\bigwedge\nolimits^{\!\bullet} L) ] \subseteq \Gamma( TY^{0} \cdot \bigwedge\nolimits^{\!\bullet} L). $$

Following Proposition \ref{prop:gms} and Remark \ref{rem:bracketonlagrangian}, we calculate the bracket on $TX$ with respect to the complement $graph(\pi)$. 
For $x,y \in \Gamma(TX)$, $$[x,y]_{H} = [x,y] + \iota_{y} \iota_{x} H,$$ and so, $$p_{TX} ([x,y]_{H}) = [x,y] - \pi^{\#} \iota_{y} \iota_{x} H.$$

Again, following Proposition \ref{prop:gms} and Remark \ref{rem:bracketonlagrangian}, this lets us calculate $d_{TX}$ on $\Gamma(\bigwedge^{\bullet} L)$: for $\alpha + \pi^{\#} \alpha \in \Gamma(L)$, $x,y \in \Gamma(TX)$
\begin{align*}
    d_{TX}(\alpha + \pi^{\#} \alpha)(x,y) &= \mathcal{L}_{x} \alpha(y) - \mathcal{L}_{y} \alpha(x) - \langle\alpha + \pi^{\#} \alpha, [x,y]_{H}\rangle \\ &= \mathcal{L}_{x} \alpha(y) - \mathcal{L}_{y} \alpha(x) - \alpha([x,y])+ H(x,y, \pi^{\#} \alpha) \\  &=d_{dR} \alpha (x,y) + H(x,y, \pi^{\#} \alpha). 
\end{align*}
and $  d_{TX}$ is extended to 
$\Gamma(\bigwedge^{\bullet} L)$ by the Leibniz rule. Since $\pi$ vanishes along $Y$, we have that for any $\omega \in \Gamma(\bigwedge^{2} L)$,
\begin{equation}\label{eq:derhamagrees} 
    \iota_{TY}^{*} d_{TX} \omega =  d_{dR} \iota_{TY}^{*} \omega.
\end{equation}

Now, assume that $N^{L} = 0$, so $$N_{TX} = (H + \pi^{\#} H) = d_{TX} \omega + \sum_{i} [l_{i}, l'_{i} ] $$ for some $\omega \in \Gamma(\bigwedge^{2} L)$, $l_i, l'_i \in \Gamma(\bigwedge^{\bullet} L)$. Using the pairing to treat both sides as sections of $\bigwedge^{\bullet}(TX)^{*}$ and restricting to $TY$ we get:
$$i^{*}_{TY} H = i^{*}_{TY} d_{TX} \omega + \sum_{i} i^{*}_{TY}[l_{i}, l'_{i} ] .$$
The first term is $i^{*}_{Y} d_{dR} \omega = d_{dR} i^{*}_{TY} \omega $ by Equation \eqref{eq:derhamagrees}, and the second is zero since all of these are brackets so they lie in $\Gamma( TY^{0} \cdot \bigwedge^{\bullet} T^{*} )$. Overall we get $i^{*}_{TY} H = d_{dR} i^{*}_{TY} \omega$, which is a contradiction, since we assumed the class of $H$ in $H^{3}_{dR} (Y)$ is nonzero. 
\end{proof}

A simple example for Proposition \ref{poissonstrvanish} is the following.
\begin{example}
    Let $(X_{1}, \pi_{1})$ be a Poisson manifold with $\pi_{1}$ vanishing at some point $pt \in X_{1}$. Let $X_{2}$ be a manifold with $H^{3}(X_{2}) \neq \{0\}$ and let $H_{2}$ be some closed $3-$form on $X_{2}$ with $[H_{2}]_{dR} \neq 0$. Extend $\pi_{1}$ to a Poisson structure $$\pi = \pi_{1} \oplus \{0\}$$ on $X_1 \times X_2$,  and pull back $H_2$ along the projection $X_1 \times X_2 \to X_2$ to a closed $3$-form $H$ on the product $X_{1} \times X_{2}$. Then $\pi$ is a twisted Poisson structure on $X_1 \times X_2$ with respect to $H$ (since $\pi^{\#} H =0$), and the submanifold $Y = \{pt\} \times X_{2}$ satisfies the conditions of Proposition \ref{poissonstrvanish}. Therefore the Dirac structure $L= graph(\pi)$ in the Courant algebroid $(TX \oplus T^{*} X)_{H}$ has $N^{L}\neq 0$.
\end{example}

Similarly to the `over a point' case (see Example \ref{ex:so3diagonal}), in the twisted exact Courant algebroid $(TX \oplus T^{*}X)_{H}$ there   exist Dirac structures $L$ with $N^{L}=0$ that do not admit Dirac complements. We show this in Proposition \ref{productnocomp},
considering the Dirac structure $L$ encoding a rank-$1$ foliation.
We first need a lemma.

\begin{lemma}\label{lemma:badleafnocomp}
    Let $L$ be a Dirac structure in $(TX \oplus T^{*}X)_{H}. $ Assume that every lagrangian complement $M$ to $L$ with integrable anchor $ F=p_{TX} (M)$ has a leaf $\sigma$ with $[i^{*}_{\sigma} H ]_{dR} \neq 0 $ (where $i: \sigma \to X$ is the immersion of $\sigma$ into $X$). Then $L$ has no Dirac complement.
\end{lemma}
\begin{proof}
    Recall that in the representation $L(F,\varepsilon)$ of a Dirac structure (as in Proposition \ref{prop:representationoflagrangian}), integrability is equivalent to $F$ being an involutive distribution and $d_{F} \varepsilon + i_{F}^{*} H=0$ (by Remark \ref{rem:integrabilityinLeepsilonrep}). Focusing on a leaf $\sigma$, we have $d_{dR} \varepsilon + i_{\sigma}^{*} H=0$. If every complement has a leaf with $[i_{\sigma} H ]_{dR} \neq 0 $, this condition is never satisfied. 
\end{proof}

\begin{prop}\label{productnocomp}
    Let $X$ be a simply connected manifold, $H \in \Omega^{3}(X)$ with $[H]_{dR} \neq 0 $. Consider the Courant algebroid $(TX' \oplus T^{*} X')_{H'}$ for $X' = X \times \mathbb{R}$ and $H'=p_{X}^{*} H$, where $p_{X}: X \times \mathbb{R} \to X$ is the projection. Let $L = T^{*}X + T\mathbb{R}$. Then $L$ has no Dirac complement, yet $N^{L}  =0$. 
\end{prop}
\begin{proof}
    We will show that if $L$ has a lagrangian complement with integrable anchor, the leaves of the anchor will be homotopic to $X \times \{0\}$, and then use Lemma \ref{lemma:badleafnocomp}. Let $M  = TX + T^{*} \mathbb{R}$ be a lagrangian complement. Any other complement of $L$ will be given by the graph of some $\omega \in \Gamma(\bigwedge\nolimits^{\!2} L)$, and we note $$\bigwedge\nolimits^{\!2} L = \bigwedge\nolimits^{\!2} T^{*}X + T^{*}X \otimes T \mathbb{R} + \bigwedge\nolimits^{\!2} T\mathbb{R} =  \bigwedge\nolimits^{\!2} T^{*}X + T^{*}X \otimes T \mathbb{R}. $$ Therefore, the anchor of $graph(\omega)$ is given by $graph(\alpha) = \{ A + \iota_{A} \alpha\ \; | \; A \in TX \}$, where $\alpha$ is the component of $\omega$ in $\Gamma(T^{*}X \otimes T \mathbb{R})$. From this we see that $\alpha$ can also be considered as a connection for the trivial fiber bundle $X \times \mathbb{R}$. 
    Assume $L$ admits a Dirac complement, so there exists $\omega \in \Gamma(\bigwedge^{2}L)$ such that $graph(\omega)$ is Dirac. Then $graph(\omega)$ would be a lagrangian complement with an integrable anchor. Since the anchor is integrable, $graph(\alpha)$ is a flat connection. Since $X$ is simply connected, this connection yields a trivialization of the fiber bundle. Namely, the projection $p_{X}: X \times \mathbb{R} \to X$ restricts to a diffeomorphism on each of the horizontal leaves of $\alpha$, which are the leaves of the anchor of $graph(\omega)$. Since $i_{X}^{*} p_{X}^{*} H = H$, and $[H]_{dR} \neq 0 $, the conditions of Lemma \ref{lemma:badleafnocomp} are satisfied, and we get a contradiction. 

    To see why $N^{L} = 0$, let $t$ be the coordinate along $\mathbb{R}$ in $X \times \mathbb{R}$ and note that $H' = [\frac{\partial}{\partial t}, tH']$, so $H' \in [\Gamma(\bigwedge^{\bullet} L), \Gamma(\bigwedge^{\bullet} L)] $. 
\end{proof}

\section{Dirac structures on Lie groups and other obstructions}\label{sec:diracstructuresonliegroups}

We deal now with the case of a Lie group $G$ with Lie algebra $\mathfrak{g}$. We show two instances in which a linear-algebraic study gives information about the, possibly twisted, Courant algebroid $TG\oplus T^*G$. 

\subsection{The Courant algebroid $(\mathfrak{g} \oplus \mathfrak{g}^{*})_{H}$}

First, for $H\in \bigwedge^3 \mathfrak{g}^*$  a Lie algebra 3-cocycle, one can  consider the exact Courant algebroid $TG\oplus T^*G$ twisted by the left-invariant extension of $H$. By left translating, we have a bijection:
    $$ \{  \: \text{Dirac  structures    on  }  (\mathfrak{g} \oplus \mathfrak{g}^{*})_{H}  \: \} \leftrightarrow \{  \: \text{Left-invariant   Dirac structures on  $(TG \oplus  T^{*}G)_{{H}}$} \}.$$
    Thus, if a Dirac structure in $(\mathfrak{g} \oplus \mathfrak{g}^{*})_{H}$ admits a Dirac complement, then their left-invariant extensions are complementary Dirac structures in  $(TG \oplus  T^{*}G)_{{H}}$.  The converse is not  true in general. 

\begin{example}\label{ex:no-complement}
     Consider $\mathfrak{g}= \mathbb{R}^{3}$ be the abelian Lie algebra, and $(\mathfrak{g} \oplus \mathfrak{g}^{*})_{H}$ for $H=dx\wedge dy \wedge dz$. Then $\mathfrak{g}^{*}$ admits no Dirac complement. However, $T^{*} \mathbb{R}^3$ has a Dirac complement, namely $graph( xdy \wedge dz)$, which is notably not left-invariant.
     \end{example}
     \begin{remark}
    Even in the compact case, there is still no stronger relation, due to the simple fact that an averaging argument, similar to the one used in the calculation of the cohomology of a Lie group, is usually not compatible with the Maurer-Cartan equation (simply put, $[\int_{G} \omega, \int_{G} \omega ] \neq \int_{G} [\omega,\omega]$).
\end{remark}

\subsection{The Courant algebroid $\mathfrak{g}\oplus \bar{\mathfrak{g}}$}
Now assume that $\mathfrak{g}$ is a quadratic Lie algebra with pairing $B$, and recall the Courant algebroid $\mathfrak{g}\oplus \bar{\mathfrak{g}}$ from Definition \ref{defi:gplusgbar}. 
There is a remarkable isomorphism 
\begin{equation}\label{eq:isopure}
 (TG\oplus T^*G)_{{\Omega}}\cong (\mathfrak{g}\oplus \bar{\mathfrak{g}})\times G   
\end{equation}
 for a natural Courant algebroid structure on the right-hand side (\cite[Sec. 3]{alekseev2007pure}), where $\Omega$ is chosen to be the left translation of the Cartan-Dirac $3$-form on $\mathfrak{g}$ (given by $\Omega(v,w,z):=\frac{1}{12}B(v,[w,z])$ for $v,w,z\in\mathfrak{g}$). Lagrangian subalgebras $\mathfrak{l}\subset\mathfrak{g}\oplus \bar{\mathfrak{g}}$ give rise, via $\mathfrak{l}\times G$, to Dirac structures in $(TG\oplus T^*G)_{\Omega}$ that are, generally, not left-invariant. For $\mathfrak{l}=\Delta$, the corresponding Dirac structure is called the Cartan-Dirac structure (first introduced in \cite{vsevera2001poisson}, and whose leaves are the conjugacy classes of $G$).

In general, it is not known when the diagonal $\Delta$ in $\lieg \oplus \bar\lieg$ admits a Dirac complement. 
We first mention two important families of examples for which $\Delta$ admits a Dirac complement.
\begin{example}\label{ex:EL1}\label{ex:doubleofquadratic}
  Let $\mathfrak{g}$ be a \emph{complex} semisimple Lie algebra  with Killing form $B$. Then $\mathfrak{g}$, viewed as a real Lie algebra and together with the pairing $2 Im(B)$, forms a quadratic Lie algebra. A lagrangian subalgebra is given by any compact real form $\mathfrak{k}$. For instance, for $\mathfrak{sl}_n(\mathbb{C})$ one can take   $\mathfrak{k}=\mathfrak{su}_n$.
Then $\mathfrak{k}$ admits a complementary lagrangian subalgebra, namely  $C:=\mathfrak{a}\oplus \mathfrak{n}$ for any Iwasawa decomposition $\mathfrak{g}=\mathfrak{k}\oplus \mathfrak{a}\oplus \mathfrak{n}$ (\cite[Ex. 2.1]{LuEvens1}). In this case, $\mathfrak{k} \oplus C:=\{(k,c):k\in \mathfrak{k}, c\in C\}$ is a lagrangian subalgebra of $\lieg \oplus \bar{\lieg}$
 transverse to $\Delta$. 
\end{example}

\begin{example}\label{ex:complement-Gauss-Dirac}\cite[Sec. 3.6]{alekseev2007pure}
    Let $\mathfrak{g}$ be a \emph{complex} semisimple Lie algebra with  Killing form $B$. Then $\mathfrak{g}$, viewed as a real Lie algebra and together with the pairing $2 Re(B)$ (which is equal to the Killing form of $\lieg$ considered as a \emph{real} Lie algebra), forms a quadratic Lie algebra. A lagrangian subalgebra complementary to the diagonal $\Delta$ is given by
    $$ \mathfrak{s} = \{ (t + \xi^+, -t + \xi^-) \in \lieg \oplus \bar{\lieg} \; | \; t \in \mathfrak{h}, \xi^{\pm} \in \mathfrak{n}^{\pm} \},$$
    where $\lieg = \mathfrak{n}^+ \oplus \mathfrak{h} \oplus \mathfrak{n}^-$ is any triangular decomposition of $\lieg$. 
\end{example}

\begin{remark}
Assume the setting of Example \ref{ex:complement-Gauss-Dirac}.
Via the Courant algebroid isomorphism \eqref{eq:isopure},
$\mathfrak{s} \times G$ defines a 
  twisted Dirac structure on $G$, called the Gauss-Dirac structure.
  It is a holomorphic Dirac structure, complementary to the Cartan-Dirac structure, with the property of having an open dense leaf containing the identity, given by the Gauss cell. 
It remains an open question to see, in the case that $\Delta$ does not admit a Dirac complement (which is possible, as we show below in Theorem \ref{theorem:diagonalcompiffabelian}), whether or not the Cartan-Dirac structure could admit a Dirac complement.  
\end{remark}

Examples \ref{ex:doubleofquadratic} and \ref{ex:complement-Gauss-Dirac} show that the Cartan-Dirac structure on $G$ 
admits a Dirac complement for either the real or the imaginary part of the Killing form. In contrast to the complex case, the question of existence of a Dirac complement to the diagonal $\Delta \subseteq \lieg \oplus \bar{\lieg}$ for a real Lie algebra $\lieg$ equipped with a definite pairing $\pairing$ is much more rigid. To see why, we first prove a classification of all Dirac structures in $\lieg \oplus \bar{\lieg}$.
\begin{lemma}\label{lemma:diracstructuresingplusgbar}
    Assume the pairing $\pairing$ on $\lieg$ is definite. Then there is a bijection:
    $$ \{ \text{Dirac structures in $\lieg \oplus \bar{\lieg} $} \} \leftrightarrow \{\text{Orthogonal automorphisms of $(\lieg, \pairing) $} \}.$$
    
\end{lemma}
\begin{proof}
Lemma \ref{lemma:lagrangiansareorthogonalbundlemaps} for the generalized metric  $(\lieg \oplus \{0\}, \{0\} \oplus \bar\lieg )$ gives a correspondence between lagrangian subspaces and orthogonal maps $\mathfrak{g} \to \mathfrak{g}$. Since $\forall x,y\in \lieg$
\[  [(x,\phi(x)),(y,\phi(y))]=([x,y],[\phi(x),\phi(y)]),
\]

the lagrangian subspace is  Dirac if and only if the orthogonal map is an automorphism. 
\end{proof}

\begin{remark}
   For $\pairing$ the Killing form, so $\lieg$ is  compact semisimple, the bijection becomes 
   $$ \{ \text{Dirac structures in $\lieg \oplus \bar{\lieg} $} \} \leftrightarrow \text{Aut}(\lieg).$$
\end{remark}

\begin{remark} 
    The identification of Lemma  \ref{lemma:diracstructuresingplusgbar} is a group isomorphism, where the group structure on the space of Dirac structures in $\lieg \oplus \bar{\lieg}$ is given by the composition of Lagrangian relations. To see why, recall the set-theoretic fact that the composition of relations given by the graphs of functions agrees with the relation given by the graph of the composition of the functions.
\end{remark}

We can now prove our final result. 

\begin{theorem}\label{theorem:diagonalcompiffabelian}
    Let $(\lieg,\pairing)$ be a quadratic Lie algebra such that $\pairing$ is definite. Then the diagonal $\Delta \subseteq \lieg \oplus \bar{\lieg}$ admits a Dirac complement if and only if $\lieg$ is abelian. 
\end{theorem}
\begin{proof}
    First, $\lieg$ must be compact, since a Lie algebra admits a nondegenerate  definite invariant pairing if and only if it is compact \cite[Ch. IX, Prop. 1]{bourbaki2008lie} (where the term \emph{\quo{separating}} is used for nondegenerate). Furthermore, compact Lie algebras are reductive \cite[Ch. IX, Prop. 1]{bourbaki2008lie}.
    
    Now, assume $\Delta$ admits a Dirac complement $M$. By Lemma \ref{lemma:diracstructuresingplusgbar}, it is given by the graph of an automorphism $\phi : \lieg \to \lieg$. Note that 
    $$\{ 0 \} =M\cap \Delta=graph(\phi) \cap graph(Id) = \{(x,x) \; | \; x\in \lieg, x = \phi(x) \},$$
    so $\phi$ must be a fixed-point-free automorphism. However, a Lie algebra admitting a fixed-point-free automorphism must be solvable by \cite[Thm. 9]{jacobson1989note}. Therefore, since $\lieg$ is both reductive and solvable, it must be abelian \cite[Ch. I, Prop. 5]{bourbaki1975lie} (otherwise, if its semisimple part is nontrivial, its derived series would stabilize at a nonzero term).
    
    Conversely, if $\lieg$ is abelian, any Lagrangian complement to $\Delta$ (which exists by Proposition \ref{prop:existencelagcomp}) is a Dirac complement. 
\end{proof}

The indefinite case, as one might expect, is more complicated due to the lack of a simple identification as in Lemma \ref{lemma:diracstructuresingplusgbar}, and so remains a target for further research. 

The following corollary stands in contrast with the complex semisimple case.

\color{black}

\begin{corol}\label{corol:cpctsemisimplenogo}
    Assume $\lieg$ is a real, compact, semisimple Lie algebra, equipped with its Killing form $B$. Then the diagonal $\Delta$ in $\lieg \oplus \bar{\lieg}$ does not admit a Dirac complement. 
\end{corol}

\begin{proof}
    The assumptions on $\lieg$ lead to the fact that $B$ is (negative) definite, and so by Theorem \ref{theorem:diagonalcompiffabelian}, $\Delta$ does not admit a Dirac complement (since there are no abelian semisimple Lie algebras). 
\end{proof}

\begin{example}
    The case of $\mathfrak{so}(3)$,  Example \ref{ex:so3diagonal}, follows now from  Corollary \ref{corol:cpctsemisimplenogo}. 
\end{example}

\begin{remark}
   The two Courant algebroids we have considered in this section are particular cases of a general construction of ``action'' Courant algebroids. 
    Let $\mathfrak{e}$ be a quadratic Lie algebra, $X$ a manifold, and  $\rho: \mathfrak{e} \to \mathfrak{X}(M)$   a Lie algebra homomorphism having coisotropic stabilizers. Then there is a natural Courant algebroid structure on $\mathfrak{e} \times X$, whose anchor at $p\in X$ is exactly $\rho_p$, and whose bracket of constant sections is the bracket on   $\mathfrak{e}$
    \cite[Thm. 2.12]{CAPoisgeom}. 
 The first case corresponds to the usual projection $(\lieg \oplus \lieg^*)_{H} \to \lieg$ composed with left translation $\lieg \to \mathfrak{X}(G)$.
    The second case corresponds, 
    for  a quadratic Lie algebra $\lieg$,
    to $(p_L)^L - (p_R)^R : \lieg \oplus \bar{\lieg} \to \mathfrak{X}(G)$ (left translation of the projection of the left component, minus right translation of the projection of the right component) \cite{semenov1985dressing}. 
    In general, any Dirac structure $\mathfrak{l}$ in $\mathfrak{e}$ now defines a Dirac structure $L = \mathfrak{l} \times G \subseteq \mathfrak{e} \times G$, and a Dirac complement to $\mathfrak{l}$ induces a Dirac complement to $L$.

\end{remark}
\color{black}

\appendix

\section{Extension of the  class $N^L$ to a general curved $L_{k}$ algebra}\label{sec:linfinityalgebras}

In this appendix we extend the construction of the obstruction class from a curved DGLA {(as in Section \ref{sec:acohomologyfordirac})}  to a general curved  $L_{n}$ algebra, where it also provides an obstruction to the existence of a Maurer-Cartan element. We first recall the definition of a curved $L_{\infty}$ algebra as in \cite{kontsevich2003deformation}.

\begin{defi}\label{eq:MC-Ln}
    A curved $L_{\infty}$ algebra is a tuple $(\mathfrak{a}, \{ \ell_k \}_{k \in \mathbb{N}})$ consisting of
    \begin{enumerate}
        \item a graded vector space $\mathfrak{a}$,
        \item and a sequence $\ell_{0}, \ell_{1},\ell_{2},...$ of degree $2-k$ maps $\ell_{k}: \mathfrak{a}^{\otimes k} \to \mathfrak{a}$, called the {brackets}, which are {graded-antisymmetric}, that is, for any homogeneous $v_{1},..., v_{k} \in \mathfrak{a}$,
        $$ \ell_{k}(v_{1},...,v_{i},...,v_{j},...,v_{k})=-(-1)^{|v_{i}| |v_{j}|}\ell_{k}(v_{1},...,v_{j},...,v_{i},...,v_{k}),$$ 
    \end{enumerate} 
    satisfying the  \textbf{higher Jacobi identity}: for any homogeneous $v_{1},...,v_{n} \in V$,
\begin{equation}\label{eq:curvedjac}
        \sum_{i+j=n+1} \;\sum_{\sigma \in \textup{UnShuff}(i,j-1)} \chi(\sigma;v_{1},...,v_{n}) \ell_{j} (-1)^{i(n-i)}(\ell_{i}(v_{\sigma(1)},...,v_{\sigma(i)}),v_{\sigma(i+1)},...,v_{\sigma(n)}) = 0, \end{equation}
    where $\chi(\sigma;v_1,...,v_n)=-(-1)^{|v_i||v_j|}$ for a transposition $\sigma$ that swaps $v_i$ and $v_j$ and is extended homomorphically to all $S_n$.  A curved \textbf{$L_n $ algebra} is a curved $L_{\infty}$ algebra with  $\ell_{k} =0$ for any $k > n$, whereas a \textbf{flat} $L_{\infty}$ algebra (or simply an $L_{\infty}$ algebra) is a curved $L_{\infty}$ algebra with $\ell_0 = 0$.

\end{defi}


\begin{remark}
  There is a notion of curved $L_{\infty}[1]$ algebra, in which all maps are graded symmetric and have degree one. For a (curved) $L_{\infty}[1]$ algebra structure on $\mathfrak{a}[1]$,  one obtains Definition \ref{eq:MC-Ln} by applying the d\'ecalage isomorphism $(\wedge^n \mathfrak{a})[n] \cong S^n(\mathfrak{a}[1])$, given by  \[v_1\cdots v_n \mapsto v_1\cdots v_n\cdot (-1)^{(n-1)|v_{1}|+\dots+2|v_{n-2}|+|v_{n-1}|} \] where $|v_i|$ denotes the degree of $v_i\in V$, 
and changing the sign of brackets of even arity  (see, for instance, \cite[\S 2]{MerkLI}, up to a sign $(-1)^n$).
\end{remark}


An example of a curved $L_\infty$ algebra is given by a curved DGLA $(\mathfrak{a},R,\ell_1,\ell_2)$, as in Definition \ref{def:curveddgla}, by considering $\ell_0=-R$, $\ell_i$ as defined, and $\ell_k=0$ for all $k>2$. Note that $\ell_0=-R$ makes the higher Jacobi identity  \eqref{eq:curvedjac} consistent with the defining properties of a curved DGLA.

\begin{remark}\label{rem-gms18}
     In relation to Remark \ref{rem:sign-change}, given transverse lagrangian subbundles $L$ and $M$, the  curved $L_{\infty}$-algebra described in \cite[Cor. 3.7]{10.1093/imrn/rny134} consists of $\ell_0=N_M$,  $\ell_1=d_M$, and $\ell_2$ the Schouten bracket associated to the projection of the Courant bracket onto $\Gamma(L)$. Note that, for example,
if $\pi$ is a bivector field on $X$, for $M=\{\pi^{\sharp}\xi + \xi \; | \; \xi \in T^*X\}$ and $L=TX$, we have $N_M=\frac{1}{2}[
\pi,\pi]\in \Gamma(\wedge^3 L)$ by \cite[Prop. 2.11]{crainic2021lectures} and $d_{\pi}=[\pi,\cdot]$, so $d_{\pi}^2=\frac{1}{2}[[\pi,\pi],\cdot]$. Thus, $[N_M,\cdot]=d_{\pi}^2$, which would be incompatible with the higher Jacobi identity  \eqref{eq:curvedjac} for $n=2$,  given by $\ell_2\ell_0+\ell_1^2=0$.  This is fixed by defining $\ell_0 = -N_M$ instead, as we have just done by choosing $\ell_0=-R$. This choice is also consistent with the case of a DGLA in \cite[Thm. 6.1]{liu-weinstein-xu:1997}.
\end{remark}


\begin{defi}
    Let $(\mathfrak{a},\ell_0,\ell_1,...,\ell_n)$ be a curved $L_n $ algebra. A \textbf{Maurer-Cartan} element is a degree $1$ element $\omega$ such that:
    $$0=\sum_{k=0}^{n}  \varsigma(k)\frac{1}{k!} \ell_k (\omega,...,\omega) = -\ell_0+\ell_1(\omega)+\frac{1}{2}\ell_2(\omega,\omega)-\frac{1}{3!}\ell_3(\omega,\omega,\omega)-\dots $$
    where $\varsigma(k):=-(-1)^{\frac{k(k+1)}{2}}$.
    
\end{defi}

A curved $L_n$ algebra furthermore admits the operation of twisting by general (not necessarily Maurer-Cartan) degree $1$ elements $\phi \in \mathfrak{a}$, in a procedure which results in a new curved $L_n$ algebra.
\begin{defi}\label{defi:twistofcurveddgla}
    Let $(\mathfrak{a},\ell_0,...,\ell_n)$ be a curved $L_n$ algebra, and $\phi \in \mathfrak{a}$ a degree 1 element. Then for
\begin{equation}\label{eq:twisting-Ln-algebra}
    \ell_k^\phi(\; \cdot \;,\ldots,\;\cdot\;) :=\varsigma(k) 
\sum_{i=k}^n \varsigma(i)\frac{1}{(i-k)!}\ell_i(\underbrace{\phi,\ldots,\phi}_{i-k \; times}, \; \cdot \;,\ldots,\; \cdot \;)
\end{equation}
    we have that $(\mathfrak{a},\ell_0^\phi,...,\ell_n^\phi)$ is a curved $L_n$ algebra, which we refer to as \quo{the twist of $\mathfrak{a}$ by $\phi$} and denote by $\mathfrak{a}^\phi$.

\end{defi}
\begin{remark}
Note that $\phi \in \mathfrak{a}^1$ is a Maurer-Cartan element in the sense of Definition \ref{eq:MC-Ln} if and only if the twist by $\phi$ is a flat $L_n$ algebra.   
\end{remark}

Remark \ref{remark:differentcomptstwist} 
 allows us to extend the construction of the cohomology in Section \ref{sec:acohomologyfordirac} to a general curved $L_{n}$ algebra. We will again show that given a curved $L_{\infty}$ algebra, we have an associated cohomology along with a distinguished cohomology class $N$ which are both invariant under twisting.
We will show  that when the $L_{\infty}$ algebra admits a Maurer-Cartan element, we have $N=0$. For $k=2$ we recover the case of curved DGLAs, as presented in Proposition \ref{prop:hownijenhuischanges}.
\\

Let    $(\mathfrak{a},\ell_{0},\ell_{1},\ell_{2},...,\ell_{n})$ be a curved $L_{n}$ algebra. In the spirit of Lemma \ref{lemma:onquotientthereisdifferential}, consider the space $$ \mathfrak{a}_{\textup{ab}} = \mathfrak{a} / \mathfrak{a}',$$
where $\mathfrak{a}'= Span( \ell_{2}(\mathfrak{a},\mathfrak{a}) + \ell_{3}(\mathfrak{a},\mathfrak{a},\mathfrak{a}) + ... + \ell_{n}(\mathfrak{a},...,\mathfrak{a}))$. That is, we quotient out by \textit{finite sums} of degree $\geq 2$ brackets of elements of $\mathfrak{a}$ like, for example, 
$ \ell_{2} (v_{1},v_{2}) + 7\ell_{3}(v_{3},v_{4},v_{5}) +...$
for $v_i \in \mathfrak{a}$.
\begin{prop}

Let $p:\mathfrak{a}\to \mathfrak{a}_\textup{ab}$ be the projection to the quotient. 
    \begin{enumerate}
    \item The bracket $\ell_{1}$ descends to an operator $d$ on 
$\mathfrak{a}_\textup{ab}$. 
    \item The operator $d$ is a differential, that is, $d^{2}=0$. 
    \item The projection $p(\ell_{0})$ of $\ell_{0}$ to the quotient is closed under $d$. 
    \item  Denote by $\mathfrak{a}^{\phi}$ the curved $L_{\infty}$-algebra twisted by a degree $1$ element $\phi \in \mathfrak{a}$.
    Then $\mathfrak{a}'=(\mathfrak{a}^{\phi})'$, and both the operator $d$   on the quotient and the class of $p(\ell_{0})$ are independent the twist.

\end{enumerate}
\end{prop} 
\begin{proof}
    \begin{enumerate}    
\item By using the fact that $\textup{UnShuff}(n,0)=\{ Id \}$, we have
     \begin{multline*}\sum_{\substack{i+j=n+1 \\ j>1}} \; \sum_{\sigma \in \textup{UnShuff}(i,j)} \chi(\sigma,v_{1},...,v_{n})(-1)^{i(n-i)} \ell_{j}(\ell_{i}(v_{\sigma(1)},...,v_{\sigma(i)}),v_{\sigma(i+1)},...,v_{\sigma(n)}) \\ = (-1)^{n}\ell_{1}(\ell_{n}(v_{1},...,v_{n}))
    \end{multline*} 
        so  the higher Jacobi identity \eqref{eq:curvedjac} gives that   $\ell_1$ maps  $\ell_{n}(v_{1},...,v_{n})$ to an element of $\mathfrak{a}'$. Therefore, we have that $\ell_{1} (\mathfrak{a}')\subseteq \mathfrak{a}'$, so $\ell_{1}$ descends to an operator $d$ on $\mathfrak{a}_\textup{ab}$ via $p \circ \ell_{1} = d \circ p$.
        \item Since $\ell_{1}^{2} = -\ell_{2}(\ell_{0},  \cdot \;)$, 
        we have that $d^{2} (p(x))= p(\ell_{2}(\ell_{0},x))=0$. 
        \item Since $\ell_{1}(\ell_{0})=0$ we have $d(p(\ell_{0}))=p(\ell_{1}(\ell_{0}))=0$.
        \item  Let $\phi \in \mathfrak{a}$ be of degree 1 and recall Equation \eqref{eq:twisting-Ln-algebra} in Definition \ref{defi:twistofcurveddgla}. 
        We show that $(\mathfrak{a}^\phi)'=(\mathfrak{a})'$. For any $v_1,...,v_k \in \mathfrak{a}$, we have
        $$ \ell^{\phi}_k(v_1,...,v_k) = \varsigma(k) \sum_{i=k+1}^n \varsigma(i)\frac{1}{(i-k)!}\ell_i(\underbrace{\phi,...,\phi}_{i-k \; times}, v_1,...,v_k) \in \mathfrak{a}'$$
    and so by the definition of $\mathfrak{a}'$, we obtain that $(\mathfrak{a}^\phi)'\subseteq \mathfrak{a}'$. The same proof shows that $((\mathfrak{a}^\phi)^{-\phi})' \subseteq (\mathfrak{a}^\phi)'$, and note that $(\mathfrak{a}^\phi)^{-\phi}=\mathfrak{a}$. Therefore, $\mathfrak{a}'=(\mathfrak{a}^\phi)'$. 

    Focusing now on the brackets $\ell_0, \ell_1$, we have that
    \begin{align*}
          \ell_{0}^\phi&= \ell_{0} - \ell_{1} (\phi) - \frac{1}{2}\ell_{2} (\phi,\phi)+...\\
     \ell_{1}^\phi&= \ell_{1} + \ell_{2} (\phi, \cdot \;) - \frac{1}{2}\ell_{3}(\phi, \phi, \cdot \;) -... 
    \end{align*}
    Therefore,  
    $$p\circ \ell_{1}^\phi = p \circ \ell_{1} + p \circ \ell_{2}(\phi, \cdot \;) + ... = p \circ \ell_{1}$$
    since the rest of the terms belong to $\mathfrak{a}'$, and, by a similar argument, 
    $$p(\ell^\phi_{0})= p(\ell_{0}) +  p(-\ell_{1}(\phi) - \frac{1}{2} \ell_{2}(\phi,\phi) +...) = p(\ell_{0})- d(p(\phi)).$$

    \end{enumerate}
    \vspace{-.8cm}
    \end{proof}
\begin{prop}
    If $\mathfrak{a}$ has a Maurer-Cartan element, the cohomology class $N$ of $p(\ell_{0})$ in $H(\mathfrak{a}_{\textup{ab}},d)$ is zero. 
\end{prop}
\begin{proof}
    Indeed, if $\phi$ satisfies the Maurer Cartan equation 
    $$ 0 = -\ell_{0} + \ell_{1}(\phi) + \frac{1}{2} \ell_{2}(\phi,\phi)-...,$$
    projecting to the quotient gives 
$ 0 = p(\ell_{0}) + d(p(\phi)),$
    so $p(\ell_{0})$ is exact.\end{proof}
\section{Proof of Proposition \ref{prop:propertiesofsum}}\label{appendixoverapoint}

Let $E_{1},E_{2}$ be Courant algebroids over a point, $L_{1},L_{2}$ Dirac structures in them. Then $E=E_1 \oplus E_2$ is a Courant algebroid over a point with bracket and pairing given by the direct sums of the brackets and the pairings on $E_1, E_2$, and $L=L_1\oplus L_2$ is obviously a Dirac structure in $E$. Note that, over a point, pullback sections of the different Courant algebroids commute with each other. 

\begin{lemma}\label{lemma:bracketcommutes}
    The natural inclusions $i_{L_j}: \bigwedge^\bullet L_j \to \bigwedge^\bullet L$ and the natural projections $p_{L_j}: \bigwedge^\bullet L \to \bigwedge^\bullet L_j$ are graded Lie algebra morphisms. 
\end{lemma}
\begin{proof} 
    For the inclusions  $i_{L_j}: \bigwedge^\bullet L_j \to \bigwedge^\bullet  L$, by definition of the brackets on $\bigwedge^\bullet L $, it is enough to see it for degree $1$ elements, where it follows directly from the fact that the bracket on $E$ is given by the sum of the brackets on $E_1,E_2$. 
    
    For the projections $p_{L_j}: \bigwedge^\bullet L \to \bigwedge^\bullet  L_j$, it is due to the fact that $[\bigwedge^{\bullet}L_1,\bigwedge^{\bullet}L_2]=0$. For $$\alpha:=\alpha_1\wedge \alpha_2, \beta:=\beta_1 \wedge \beta_2  \in   \bigwedge\nolimits^{\!\bullet}(L_1 \oplus L_2),$$ 
    with $\alpha_i,\beta_i \in \bigwedge\nolimits^{\!\bullet} L_i$, we have that 
    $$p_{L_1}([\alpha,\beta])= \pm p_{L_1}([\alpha_1,\beta_1] \otimes (\alpha_2 \wedge \beta_2) ) \pm p_{L_1}((\alpha_1 \wedge \beta_1)\otimes [\alpha_2,\beta_2]  ) $$
    is zero unless $\alpha_2,\beta_2\in \bigwedge^0 L_2=\mathbb{R}$, case in which we have 
    $$ p_{L_1}([\alpha,\beta]) = p_{L_1}([\alpha_1,\beta_1]) = [\alpha_1,\beta_1] = [p_{L_1}(\alpha),p_{L_1}(\beta)] $$
    as needed. The proof for $p_{L_2}$ is completely analogous.
\end{proof} 

Now, let $M_1, M_2$ be lagrangian complements to $L_1,L_2$ in $E_1,E_2$, respectively. Then $M=M_1\oplus M_2$ is a lagrangian complement to $L_1\oplus L_2$ in $E_1 \oplus E_2$. Note that under the isomorphisms $M_j \cong L_j^*$ induced by the pairing, we have that $i_{L_j}=p_{M_j}^*$ and $p_{L_j}=i_{M_j}^*$ since $M_j \perp L_k$ for $j \neq k$. We have the following similar result for the derivations $d_{M_j}, d_M$: 
\begin{lemma}\label{lemma:pullbackofderivative}
    The inclusions $i_{L_j}: \bigwedge^\bullet L_j \to \bigwedge^\bullet L$ and the projections $p_{L_j}: \bigwedge^\bullet L \to \bigwedge^\bullet L_j$ satisfy: 
    \begin{align*}
         i_{L_j} \circ  d_{M_j} &= d_M \circ i_{L_j}, &
     p_{L_j} \circ  d_{M} &= d_{M_j} \circ p_{L_j}.
    \end{align*}
\end{lemma}

\begin{proof}
      Since $i_{L_j}, p_{L_j}$ are morphisms for the wedge product and $d_M$, $d_{M_j}$ are derivations of the wedge product, it is enough to prove this for $1$-forms, that is, elements of $L$ and $L_j$. 
    
    Note that Lemma \ref{lemma:bracketcommutes} also applies for the \quo{projected bracket}, that is, the inclusions  $ i_{M_j} : M_j \to M$ and the projections $ p_{M_j}: M \to M_j $ commute with the brackets on $M, M_j$ respectively, as defined in Remark \ref{rem:bracketonlagrangian}. Recall that for $l \in M^* \cong L$, and $m, m' \in M$:
    $$ d_{M} l(m, m') = -\langle l, p_M([m,m']\rangle$$
    and similar formulas hold for $M_1, M_2$. It immediately follows that for a $1$-form $l_j \in L_j$ we have: 
     \begin{align*}
           d_M \circ i_{L_j}(l_j)(m,m')&= -\langle l_j, p_M[m,m']\rangle = -\langle l_j, p_{M_j}(p_M[m,m'])\rangle = -\langle l_j, p_{M_j}([m,m'])\rangle\\  &= -\langle l_j, [p_{M_j}(m),p_{M_j}(m')]\rangle =  p_{M_j}^* \circ  d_{M_j}(l_j) (m,m') = i_{L_j} \circ  d_{M_j}(l_j) (m,m')
     \end{align*}
     where the second equality follows again from the fact that $M_j \perp L_k$ for $j \neq k$.
     
    The second equality follows from an analogous calculation.\end{proof}
\begin{lemma}\label{lemma:nijenhuisofsum}
The Nijenhuis tensors of $M, M_j$ satisfy
$p_{L_j}(N_M) = N_{M_j}$.
\end{lemma}
\begin{proof}
    This is directly due to the fact that the bracket and pairing on $E$ are given by the sums of the brackets and pairings on $E_1,E_2$, and recalling that $p_{L_j} = i_{M_j}^*$.
\end{proof}
\begin{corollary}\label{cor:projectioninclusioncurveddglamorphism}
    The maps $i_{L_j} : \bigwedge^{\bullet} L_j \to \bigwedge^{\bullet} L$ and $p_{L_j} : \bigwedge^{\bullet} L \to \bigwedge^{\bullet} L_j$ are curved DGLA morphisms. 
\end{corollary}

Using these lemmas, we can prove Proposition \ref{prop:propertiesofsum}.
\begin{proof}[Proof of Proposition \ref{prop:propertiesofsum}]


    \begin{enumerate}
        \item  If $M_{1},M_{2}$ are Dirac complements to $L_{1},L_{2}$, then $M_{1}\oplus M_{2}$ gives a Dirac complement to $L_{1} \oplus L_{2}$. To see the converse,
    choose $M_{1},M_{2}$ complementary lagrangian subspaces to $L_{1},L_{2}$, so $M = M_{1} \oplus M_{2}$ is a complementary lagrangian subspace to $L_{1}$. A Dirac complement to $L_{1} \oplus L_{2}$ is given by the graph of $\omega \in \bigwedge^{2} (L_{1} \oplus L_{2})$ satisfying $$N_{M} + d_{M} \omega + \frac{1}{2} [\omega,\omega] = 0.$$ Applying $p_{L_j}$ gives:
    $$ 0 = p_{L_j}(N_M) + p_{L_j} d_M \omega + \frac{1}{2}p_{L_j} [\omega,\omega] = N_{M_j} + d_{M_j} p_{L_j}\omega +\frac{1}{2}[p_{L_j} \omega, p_{L_j}\omega]$$
    using Lemmas \ref{lemma:bracketcommutes}-\ref{lemma:nijenhuisofsum}, showing that $p_{L_j}\omega$ satisfies the Maurer-Cartan equation in $\bigwedge^\bullet L_j$, therefore its graph is a Dirac complement to $L_j$.
    \item Lemma \ref{lemma:bracketcommutes} shows that the inclusions $i_{L_j} $ satisfy $i_{L_j} ([\sectionsa{L_j},\sectionsa{L_j}]) \subseteq [\sectionsa{L},\sectionsa{L}]$, so we have an induced map $(\tilde{i}_{L_1} \oplus \tilde{i}_{L_2}): \reduced{L_1}\oplus\reduceds{L_2} \to \reduceds{L} $. Lemma \ref{lemma:pullbackofderivative} shows this map descends to a map
    \begin{align*}
    \mu: \cohomology{\bullet}{L_1} \oplus \cohomology{\bullet}{L_2} & \to \cohomology{\bullet}{L}\\ 
    [\underline{\omega_1}] \oplus [\underline{\omega_2}] &\mapsto [ \tilde{i}_{L_1}\underline{\omega_1} \oplus  \tilde{i}_{L_2}\underline{\omega_2} ].   
    \end{align*}
     We show that $\mu$ is injective. Let $\omega_j \in \bigwedge^\bullet L_j$ represent classes in $H^\bullet(\reduceds{L_j})$, so $d_{M_j} \omega  \in [\sectionsa{L_j}, \sectionsa{L_j}]$. Note that if $ \mu( [\underline{\omega_1}] \oplus [\underline{\omega_2}]) =[ \tilde{i}_{L_1}\underline{\omega_1} \oplus  \tilde{i}_{L_2}\underline{\omega_2} ] =  0$, then
    $$ i_{L_1}\omega_1 + i_{L_2}\omega_2 = d_{M} \alpha + \sum_{k} [\beta_k, \beta'_k] $$
    for some $\alpha, \beta_k,\beta_k'\in \bigwedge^\bullet L$. Applying $p_{L_j}$ gives
    $$\omega_j = p_{L_j}(d_{M} \alpha + \sum_{k} [\beta_k, \beta'_k]) = d_{M_j}p_{L_j}(\alpha) + \sum_k  [p_{L_j}(\beta_k), p_{L_j}(\beta'_k)]$$
    due to Corollary \ref{cor:projectioninclusioncurveddglamorphism}, so $[\omega_j]=0$. 
    \item It follows immediately from Lemma \ref{lemma:nijenhuisofsum}. 
    \end{enumerate}
    \vspace{-0.4cm}
\end{proof}

\printbibliography\vspace{.2cm}

\end{document}